\documentclass[10pt]{amsart}

\usepackage{amscd,amssymb,amsthm,verbatim}
\usepackage{enumerate}
\usepackage{bm}
\usepackage{mathrsfs, graphicx} 
\usepackage{stmaryrd}
\usepackage{subfig,tikz}
\usetikzlibrary{decorations.pathmorphing}
\newcounter{nameOfYourChoice}
\usepackage[%
bookmarks=true,
colorlinks,
linkcolor=blue,
urlcolor=blue,
citecolor=red,
plainpages=false,
pdfpagelabels,
final,
breaklinks=true\between
]{hyperref}

\numberwithin{equation}{section}

\theoremstyle{plain}
\newtheorem{theorem}{Theorem}[section]
\newtheorem{proposition}[theorem]{Proposition}
\newtheorem{lemma}[theorem]{Lemma}

\theoremstyle{definition}
\newtheorem{definition}[theorem]{Definition}
\theoremstyle{remark}
\newtheorem{remark}[theorem]{Remark}
\DeclareFontFamily{U}{MnSymbolC}{}
\DeclareSymbolFont{MnSyC}{U}{MnSymbolC}{m}{n}
\DeclareFontShape{U}{MnSymbolC}{m}{n}{
	<-6>  MnSymbolC5
	<6-7>  MnSymbolC6
	<7-8>  MnSymbolC7
	<8-9>  MnSymbolC8
	<9-10> MnSymbolC9
	<10-12> MnSymbolC10
	<12->   MnSymbolC12}{}
\DeclareMathSymbol{\intprod}{\mathbin}{MnSyC}{'270}

\begin{document}

\title[Stable Surfaces and Free Boundary MOTS]
{Stable Surfaces and Free Boundary marginally outer trapped surfaces}
\author{Aghil Alaee}
\address{\parbox{\linewidth}{Aghil Alaee\\
	Department of Mathematics and Computer Science, Clark University, USA\\
	Center of Mathematical Sciences and Applications, Harvard University, USA}}
\email{aalaeekhangha@clarku.edu, aghil.alaee@cmsa.fas.harvard.edu}
\author{Martin Lesourd}
\address{\parbox{\linewidth}{Martin Lesourd\\
Black Hole Initiative, Harvard University, USA}}
	\email{mlesourd@fas.harvard.edu}
\author{Shing-Tung Yau}
\address{\parbox{\linewidth}{Shing-Tung Yau\\
Department of Mathematics, Harvard University, USA}}
\email{yau@math.harvard.edu}
\begin{abstract}
We explore various notions of stability for surfaces embedded and immersed in spacetimes and initial data sets. The interest in such surfaces lies in their potential to go beyond the variational techniques which often underlie the study of minimal and CMC surfaces. We prove two versions of Christodoulou-Yau estimate for $\mathbf{H}$-stable surfaces, a Cohn-Vossen type inequality for non-compact stable marginally outer trapped surface (MOTS), and a global theorem on the topology of $\mathbf{H}$-stable surfaces. Moreover, we give a definition of capillary stability for MOTS with boundary. This notion of stability leads to an area inequality and a local splitting theorem for free boundary stable MOTS. Finally, we establish an index estimate and a diameter estimate for free boundary MOTS. These are straightforward generalizations of Chen-Fraser-Pang and Carlotto-Franz results for free boundary minimal surfaces, respectively. 
\end{abstract}

\maketitle

\section{Introduction}
The study of minimal hypersurfaces or more generally hypersurfaces with \textit{constant mean curvature} (CMC) in a Riemannian manifold $(M,g)$ has a natural generalization motivated by general relativity: the study of surfaces with prescribed \textit{null mean curvature} in \textit{initial data sets} and \textit{spacetimes}. 

A \textit{spacetime} is an $4$-dimensional time-oriented Lorentzian manifold $(L^{4},h)$ that satisfies the Einstein field equation
\begin{equation}
G_h\equiv \text{Ric}_{h}-\frac{1}{2}R_h h=T,
\end{equation}
where $\text{Ric}_{h}$ and $R_h$ are Ricci and scalar curvature of metric $h$ and $T$ is a symmetric two tensor associated with matter sources in $L^{4}$. By Gauss and Codazzi equations, we consider an \textit{initial data set} $(M^{3},g,k)$, where $(M^3,g)$ is a $3$-dimensional Riemannian manifold and $k$ is the second fundamental form, that satisfies the Einstein constraint equations
\begin{equation}\label{constraint}
R_g+(\text{tr}_g(k))^2-|k|^2=2\mu ,\qquad \text{div}_g(k-\text{tr}_g(k)g) =J,
\end{equation} where $\mu=T(\tau, \tau)$ and $J=T(\tau, \cdot)$ are the local energy density and momentum density and $\tau$ is the unit future-directed timelike normal vector on $(M^3,g,k)$.
 
In this paper, we study two kinds of objects: surfaces $\Sigma^2$ that are embedded (or immersed) in $M^3$ with unit normal vector $N$ and mean curvature $H$, and surfaces $\Sigma^{2}$ embedded (or immersed) in $L^{4}$ with unit future-directed nulllike normal vectors $l_{\pm}$ and associated \textit{null mean curvatures} $\theta_+$ and $\theta_-$. An example of prescribed null mean curvature surfaces that has gathered interest in recent years are \textit{marginally outer trapped surfaces} (MOTS), see \cite{AnderssonEichmairMetzger} and references therein. These are two sided spacelike hypersurfaces embedded in $(M^3,g,k)$ that have vanishing null expansion in a chosen direction $\theta_+=0$. Despite not having constant mean curvature, MOTS behaves like minimal and CMC surfaces in certain respects.  

Closed MOTS are physically interesting because these surfaces model the boundaries of black holes in general relativity. The existence of such surfaces were first obtained in \cite{SchoenYauII} where they were shown to be obstructions to global solutions of a prescribed mean curvature equation known as Jang's equation. Based on the analysis in \cite{SchoenYauII}, closed MOTS were also shown to exist in domains with two boundary components satisfying certain convexity properties \cite{AnderssonMetzger,Eichmair}. These ideas were used to obtain a black hole existence theorem in \cite{SY83}, later refined into a hoop conjecture result in \cite{ALY19}. 

For MOTS with boundary, Eichmair \cite{Eichmair} showed that such surfaces solve a certain version of the Plateau problem. In particular, for any smooth compact domain $\Omega$ with boundary $\partial \Omega$ that is seperated by a closed embedded curve $\gamma$, that is $\partial \Omega \backslash \gamma =\partial \Omega_1 \dot{\cup} \partial \Omega_2$ where $\partial \Omega_1$ and $\partial \Omega_2$ are disjoint and relatively open, such that $\theta_+|_{\partial \Omega_1}>0$ and $\theta_-|_{\partial \Omega_2}<0$, there exists an embedded MOTS $\Sigma$ with boundary $\partial \Sigma =\gamma$ that is smooth up to and including its boundary. 

MOTS are not known to minimize any elliptic functional (though they are \textit{$C$-almost-minimizing} \cite{Eichmair} as currents), and consequently it is generally not possible to apply the usual variational techniques used to study minimal and CMC surfaces. In spite of this, a stability concept for MOTS was defined in \cite{AnderssonMarsSimon} and has found various uses in \cite{G18,AnderssonMetzger10,C16,EHLS15,EM16,GM18,GallowaySchoen}. 

Here, we investigate further the analogy between surfaces of prescribed null mean curvature and their Riemannian counterparts, and obtain the following set of results.\\ 

\noindent\textbf{1. An initial data set generalization of Christodoulou-Yau's estimate \cite{CY}}. A hypersurface $\Sigma$ in a Rieamnnain manifold $(M,g)$ which locally minimizes its area given the enclosed volume is called a \textit{stable volume preserving surface}. Christodoulou and Yau obtained an estimate for Hawking mass of stable volume preserving, topologically $\mathbb{S}^2$, hypersurface $\Sigma$ in a Riemannian manifold $(M,g)$. In particular, for these surfaces that the Hawking mass is bounded below by a constant multiple of the integral of scalar curvature of $g$ over the surface. Therefore, if the scalar curvature is positive, the Hawking mass of these surfaces are positive. This estimate plays a key role in the recent progress on the question of canonical foliations and isoperimetry in asymptotically flat Riemannian manifolds, see \cite{EM13} and various references. \\ \indent 
For a surface $\Sigma$ in an initial data set $(M,g,k)$, it is no longer clear which kind of surfaces are to be estimated. Here we consider surfaces that are $\mathbf{H}$-stable, where $\mathbf{H}$ is mean curvature vector of $\Sigma$, with respect to a given direction $X$, and prove two versions of Christodoulou-Yau's estimate for $\mathbf{H}$-stable surfaces in Theorem \ref{thmm1} and Theorem \ref{thmm2}.  

We note that the selection of these surfaces is motivated by work of Metzger \cite{M07}, Nerz \cite{N18}, and Cederbaum and Sakovich \cite{CS19}, whom together show that asymptotically flat initial data sets with non-zero energy admit a foliation with leaves having constant $\theta_+\theta_-$. We expect that an estimate of this kind will lead better understanding of the foliation in \cite{CS19}.  \\ 

\noindent\textbf{2. Global results for $\mathbf{H}$-stable surfaces and non-compact stable MOTS}. The study of stable CMC or minimal surfaces has a long history in Riemannian geometry, see \cite{MPR08} and references therein. Whether in $\mathbb{R}^3$ or in manifolds with lower bounds on their Ricci or scalar curvature, a growing collection of global results have been obtained over the years. Here, we obtain analogous results in this context. In Theorem \ref{thm4.2}, we prove a Cohn-Vossen type inequality 
\begin{equation}
	\int_{\Sigma} \left(\mu+J(N)\right)d\mu_{\Sigma} \leq 2\pi,
	\end{equation}for any complete non-compact stable MOTS $\Sigma$ with unit normal vector $N$ in an initial data set satisfying $\mu-|J|> 0$. 
	In addition, in Theorem \ref{thm4.81}, we obtain global result on topology of $\mathbf{H}$-stable surfaces which generalizes Theorem 2.13 of Meeks-Perez-Ros \cite{MPR08} obtained for stable CMC surfaces.\\

\noindent\textbf{3. A definition of stability for capillary and free boundary MOTS}. Free boundary minimal surfaces have been the subject of intense research in recent years. Analogous to the successful concept of stability for closed MOTS introduced in \cite{AnderssonMarsSimon}, we define the notion of stability for capillary and free boundary MOTS, see Definition \ref{Defcap}. The definition in \cite{AnderssonMarsSimon} was applied to MOTS with boundary in \cite{GM08,EHLS15,EM16} for variations that are compactly supported in the interior. We allow for variations of the boundary, and although the definition given here is not the most general imaginable, it is chosen so that its PDE properties yield a meaningful concept of stability; that is, one that relies on the existence of real principle eigenvalue for elliptic non-self adjoint operators with Robin boundary condition, see Theorem \ref{ThmA1}. \\ \indent
An interesting application of this kind of capillary stability for MOTS would lie in generalizing Gromov's polyhedral comparison theory for $R>0$ to the setting of the dominant energy condition $\mu\geq |J|$. In Li's proof of his polyhedral comparison theorem \cite{L19}, one minimizes an elliptic functional so as to produce a stable capillary minimal surface, and from this the result follows straightforwardly. Stable capillary MOTS, if constructed, would yield a generalized version of this result. \\

\noindent\textbf{4. Inequalities and the rigidity of stable free boundary MOTS.} A classic observation of Schoen and Yau \cite{SY79} states that a closed stable minimal hypersurface $\Sigma$ in a closed manifold $(M,g)$ with $R_g\geq 0$ is either Yamabe positive or rigid otherwise. Based on this and ideas from \cite{FS80,SchoenYauII}, Galloway and Schoen \cite{GallowaySchoen} have shown an analogous result for closed stable MOTS $\Sigma$ in initial data sets $(M,g,k)$ with $\mu-|J|\geq 0$. The result was refined in \cite{G18} and the case where $\mu-|J|\geq c>0$ was studied in \cite{GM18}. Similar rigidity results have been obtained for closed minimal surfaces in manifolds with $R>0$ by Bray, Brendle, Eichmair, and Neves in \cite{BBN10,BBEN10}, and by Nunes for $R> -c$ with $c>0$ \cite{N11}. In those works, one first obtains a local rigidity result, which is then globalized by combining various classic results with elementary topological arguments. In a similar spirit, Ambrozio \cite{A13} studied the case of stable free boundary minimal surfaces and obtained a boundary version of these aforementioned rigidity results. In Proposition \ref{prop1rigidity} and Theorem \ref{prop2rigidity}, we prove a generalization of \cite{A13} and \cite{GM18} to the case of free boundary MOTS. In particular, we prove an estimate for a quantity, in term of  area of $\Sigma$ and arc-length of $\partial\Sigma$, such that the rigidity gives us a local splitting theorem in a neighborhood of stable free boundary MOTS.

A typical object of study of free boundary minimal surfaces are inequalities that combine area, Euler characteristic, and Morse index. These are readily obtainable in domains with positive scalar or Ricci curvature with boundaries satisfying various convexity assumptions. For instance, Chen, Fraser, and Pang proved an estimate for area of a compact orientable two-sided free boundary minimal surface with index 1, genus $g$, and $l\geq 1$ boundary components in a Riemannian manifold with non-negative scalar curvature \cite{CFP12}. In Proposition \ref{indexestimate}, we show that this result generalize straightforwardly to the case free boundary MOTS in an initial data set satisfying dominant energy condition $\mu-|J|\geq 0$. Moreover, Carlotto and Franz \cite{CF19}, among other things, obtain a diameter estimate for stable free boundary minimal surfaces, from which they then obtain a useful area bound. Their proof combines an argument of Fischer-Colbrie \cite{F85}, a classical theorem of Hartman \cite{H64}, and a computation in \cite{SY83} also employed in \cite{C17}. We note here that these arguments apply virtually unchanged to yield Proposition \ref{prop6.7}, yielding a diameter estimate for stable free boundary MOTS.\\ \indent 
These two instances of straightforward generalizations from minimal surfaces to MOTS suggest the following additional open problems: (i) obtain a general existence theorem for free boundary MOTS within domains satisfying appropriate convexity conditions, (ii) obtain curvature estimates for stable free boundary MOTS generalizing \cite{AnderssonMetzger10}, and (iii) in either the closed or free boundary setting, study compactness phenomenona for MOTS.
\vspace{.5cm}

\textbf{Acknowledgements}. A. Alaee acknowledges the support of the AMS–Simons Travel Grant. M. Lesourd acknowledges the support of the Gordon and Betty Moore Foundation and the John Templeton Foundation. S.-T. Yau acknowledges the support of NSF Grant DMS-1607871.

\section{Variations of mean curvatures} \label{Sec2}
We consider spacetime $(L^4,h)$ where $L^4$ is a smooth oriented manifold equipped with Lorentzian metric $h$. Moreover, let $(M^3,g,k)$ be an initial data set in $(L^4,h)$ that satisfies the Einstein's constraint equations \eqref{constraint}, where $M^4$ is a smooth oriented Riemannian manifold with metric $g$ and $k$ is the second fundamental form of hyeprsurface $M^3$ in $L^4$. Let $\Sigma\hookrightarrow L^4$ be a spacelike 2-surface with induced metric $g_{\Sigma}$ embedded in $M^3$ and unit normal $N$ in $M^3$. The normal bundle of this surface in $L^4$ is diffeomorphic to the Minkowski space $\mathbb{R}^{1,1}$ and thus $\Sigma$ admits a smooth null future-directed basis $l_{\pm}=\tau\pm N$, where $\tau$ and $N$ are the unit timelike and spacelike normal vectors to the surface $\Sigma$ with respect to $L^4$, such that $h(l_{+},l_-)=-2$. 

Define two null second fundamental forms
\begin{equation}
\chi_{\pm}(X,Y)=g_{\Sigma}({}^L \nabla_Xl_{\pm},Y),
\end{equation}
where $X,Y$ are tangent vector fields on $\Sigma$ and ${}^L\nabla$ is the Levi-Civita connection with respect to $h$. The null mean curvature are denoted by $\theta_{\pm}=\text{tr}_{\Sigma}\chi_{\pm}$. Similarly, the second fundamental form of $\Sigma$ in $M^3$ is defined by
\begin{equation}
A(X,Y)=g_{\Sigma}({}^M\nabla_X N,Y),
\end{equation}where ${}^M\nabla$ is the Levi-Civita connection with respect to $g$ with mean curvature $H=\text{tr}_{\Sigma}A$. Define smooth function $P:=\text{tr}_{\Sigma}k$. Combining definitions of mean curvatures and $l_{\pm}=\tau\pm N$, we have  $\theta_{\pm}=\pm H+P$. The second fundamental form of $\Sigma$ in $L^4$ is defined as
\begin{equation}
\Pi(X,Y)=A(X,Y)N-k_{\Sigma}(X,Y)\tau.
\end{equation}
where $k_{\Sigma}$ is $k$ restricted to tangent bundle of $\Sigma$, and the mean curvature vector field is $\mathbf{H}=H N-(\text{Tr}_{\Sigma}k) \tau$. This implies
\begin{equation}
-\theta_{+}\theta_-=|\mathbf{H}|^2=H^2-P^2. 
\end{equation}Then we have the following Lemma.
\begin{lemma}\label{variationlemma}Let $\Sigma$ be a spacelike 2-surface in spacetime $(L^4, h)$. Consider variation of $\Sigma$ with smooth embedding map $f:\Sigma\times(-\epsilon,\epsilon)\to L^4$ such that $f(\Sigma,t)=f_t(\Sigma)=\Sigma_t$ and variational vector field is $X=\frac{\partial f_t}{\partial t}\vert_{t=0}$. Moreover, consider the smooth function $\varphi\in C^{\infty}(\Sigma)$. 
\begin{enumerate}[(a)]
	\item If $X=\varphi N$, then variation of future null mean curvature is
	\begin{equation*}
	\delta_{X}\theta_+=-\Delta \varphi+2\langle W,\nabla \varphi\rangle +\left(Q+ \text{div}\,W-|W|^2+\theta_+\text{tr}_gk -\frac{1}{2}\theta_+^2\right)\varphi,
	\end{equation*}where  $\Delta$, $\nabla$, and $\text{div}$ are with respect to induced Riemannian metric on $\Sigma$, $W(Y)=k(Y,N)$ is the connection one-form of normal bundle of $\Sigma$ for any tangent vector field $Y\in T\Sigma$, and 
	\begin{equation*}
	Q:=\frac{1}{2}R_\Sigma - \mu - \langle J,N \rangle -\frac{1}{2}|\chi_+|^2
	\end{equation*}
	 where $R_{\Sigma}$ is scalar curvature of $\Sigma$ and $\mu$ and $J$ are energy and momentum density, respectively. Furthermore, if $H>0$, the variation of norm of mean curvature vector is
	\begin{equation*}
	\begin{split}
	\frac{1}{2H}\delta_{X}|\mathbf{H}|^2=&-\Delta\varphi+\frac{1}{2}\left[R_{\Sigma}-2\mu-(\text{tr}_gk)^2 +|k|^2-|A|^2-H^2 \right] \varphi\\ 
	&-2\frac{P}{H} \langle W,\nabla \varphi\rangle-\frac{P}{H} \left[-J(N)+\text{div}\,W+Hk(N,N)-\langle A,k_{\Sigma}\rangle \right] \varphi.
	\end{split} 
	\end{equation*} 
	\item If $X=-\varphi l_-$ and $\theta_-\neq 0$, then the variation of norm of mean curvature vector is
\begin{equation*}
\frac{1}{2\theta_-}\delta_{X}|\mathbf{H}|^2=\Delta\varphi-2\langle W,\nabla \varphi\rangle -\left(\text{div}_{\Sigma}W-|W|^2+\overline{Q}\right)\varphi,
\end{equation*}
	where  $\overline{Q}=\frac{1}{2}R_{\Sigma}-\frac{1}{2}G(l_+,l_-) +\tfrac{\theta_+}{2\theta_-}\left( |{\chi}_-|^2+ G(l_-,l_-)\right) +\frac{1}{2}\theta_-\theta_+$.
\end{enumerate}
\end{lemma}
\begin{proof}(a) Consider smooth functions $\alpha,\beta\in C^{\infty}(\Sigma)$ and the variation vector field 
	\begin{equation}
	X=\frac{\partial f_t}{\partial t}\bigg\vert_{t=0}=\alpha l_+ +\beta l_-.
	\end{equation}Following \cite[Lemma 4.1]{AnderssonMetzger10}, a straightforward computation of variation of $\theta_+$ in the direction $X$ is
\begin{equation}\label{variation1}
\begin{split}
\delta_{X}\theta_+=&2\Delta\beta-4\langle W,\nabla\beta\rangle-\alpha\left(|\chi_+|^2+\text{Ric}_{L}(l_+,l_+)\right)+\theta_+\kappa_{X}\\
&-\beta\left(2\text{div}W-2|W|^2-|\Pi|^2+\text{Ric}_{L}(l_+,l_-)-\frac{1}{2}\text{Rm}_{L}(l_+,l_-,l_-,l_+)\right),
\end{split}
\end{equation}where $\text{Ric}_{L}$, and $\text{Rm}_{L}$ are Ricci and Riemann curvatures with respect to metric $h$ and $\kappa_{X}$ is a gauge dependent quantity, i.e., under rescaling $l_\pm\to e^{\pm f}l_{\pm}$, then $\kappa_X\to \kappa_X+\mathcal{L}_Xf$. Similarly, we have 

\begin{equation}\label{variation2}
\begin{split}
\delta_{X}\theta_-&=2\Delta\alpha+4\langle W,\nabla\alpha\rangle-\beta\left(|\chi_-|^2+\text{Ric}_{L}(l_-,l_-)\right)-\theta_-\kappa_{X}\\
&-\alpha\left(-2\text{div}W-2|W|^2-|\Pi|^2+\text{Ric}_{L}(l_+,l_-)-\frac{1}{2}\text{Rm}_{L}(l_+,l_-,l_-,l_+)\right)
\end{split}
\end{equation}By \cite[equation (9)]{AnderssonMetzger10}, we observe that the Gauss equation leads to 
\begin{equation}\label{Gauss}
R_{\Sigma}=R_{L}+2\text{Ric}_{L}(l_+,l_-)-\frac{1}{2}\text{Rm}_{L}(l_+,l_-,l_-,l_+)+|\mathbf{H}|^2-|\Pi|^2
\end{equation}If we set $\alpha=-\beta=\frac{\varphi}{2}$, then $X=\varphi N$ which leads to
\begin{equation}
\kappa_{X}=\kappa_{\varphi N}=\frac{\varphi}{2}k(N,N)=\frac{\varphi}{2}\left(\text{tr}_gk-P\right).
\end{equation} 
Together with Gauss equation \eqref{Gauss}, variation equation \eqref{variation1}, and Einstein equation
\begin{equation}
\text{Ric}_{L}(l_+,\tau)+\frac{R_L}{2}=G(l_+,\tau)=T(l_+,\tau)=\mu+J(N)
\end{equation}we have 
\begin{equation}\label{vart+}
\delta_{\varphi N}\theta_+=-\Delta \varphi+2\langle W,\nabla \varphi\rangle +\left(Q+ \text{div}W-|W|^2+\theta_+\text{tr}_gk -\frac{1}{2}\theta_+^2\right)\varphi,
\end{equation}where 
\begin{equation}
Q:=\frac{1}{2}R_\Sigma - \mu - \langle J,N \rangle -\frac{1}{2}|\chi_+|^2.
\end{equation}Next consider the variation of $|\mathbf{H}|$ as follows
\begin{equation}
\frac{1}{2H}\delta_{X}|\mathbf{H}|^2=\delta_{X}H-\frac{P}{H}\delta_{X}P.
\end{equation}The second variation of area function implies \begin{equation}
\delta_{X}H=-\Delta\varphi-(\text{Ric}_g(N,N)+|A|^2)\varphi
\end{equation} and a computation as in \cite[Page 242]{DanLee} leads to
\begin{equation}
\begin{split}
\delta_{X}P&=2\langle W,\nabla\varphi\rangle+\left(\nabla_{N}\text{tr}_gk-\nabla_{N}k(\nu,\nu)\right)\varphi\\
&=2\langle W,\nabla\varphi\rangle+\left(\nabla_{N}\text{tr}_gk-\text{div}_gk(N)+\text{div}W+Hk(N,N)-\langle A,k_{\Sigma}\rangle\right)\varphi\\
&=2\langle W,\nabla\varphi\rangle+\left(-J(N)+\text{div}W+Hk(N,N)-\langle A,k_{\Sigma}\rangle\right)\varphi.
\end{split}
\end{equation}Combining these with constraint equation \eqref{constraint} and Gauss equation 
\begin{equation}
\frac{1}{2}(R_{\Sigma}-R_M-|A|^2-H^2)=-\left(\text{Ric}_g(N,N)+|A|^2\right),
\end{equation}we have
\begin{equation*} 
\begin{split}
	\frac{1}{2H}\delta_{X}|\mathbf{H}|^2=&-\Delta\varphi+\frac{1}{2}\left[R_{\Sigma}-2\mu-(\text{tr}_gk)^2 +|k|^2-|A|^2-H^2 \right] \varphi\\ 
	&-2\frac{P}{H} \langle W,\nabla \varphi\rangle-\frac{P}{H} \left[-J(N)+\text{div}W+Hk(N,N)-\langle A,k_{\Sigma}\rangle \right] \varphi.
\end{split} 
\end{equation*} 
(b) If we set $\alpha=0$ and $\beta=-\varphi$, then variational vector field is $X=-\varphi l_-$. First observe that 
\begin{equation}\label{H1}
\delta_{X}|\mathbf{H}|^2=-\delta_{X}\left(\theta_+\theta_-\right)=-\theta_-\delta_{X}\theta_+-\theta_+\delta_{X}\theta_-
\end{equation}It follows from the variation equation \eqref{variation1}, Gauss equation \eqref{Gauss}, Einstein equation $G(l_+,l_-)=\text{Ric}_{L}(l_+,l_-)+R_L$ that
\begin{equation}\label{t+}
\delta_{-\varphi l_-}\theta_+=-2\Delta\varphi+4\langle W,\nabla \varphi\rangle+\kappa_{-\varphi l_-}\theta_++\varphi\left( 2\text{div}W-2|W|^2+R_{\Sigma}-G(l_+,l_-)+\theta_-\theta_+ \right).
\end{equation}Similarly, by $G(l_-,l_-)=\text{Ric}_{L}(l_-,l_-)$ we have  
\begin{equation}\label{t-}
\begin{split}
\delta_{-\varphi l_-}\theta_-&=\varphi\left(|{\chi}_-|^2+ G(l_-,l_-)\right)-\kappa_{-\varphi l_-}\theta_-\\
&=\varphi\left(|\hat{\chi}_-|^2+ G(l_-,l_-)+\frac{1}{2}\theta_-^2\right)-\kappa_{-\varphi l_-}\theta_-
\end{split}
\end{equation}where $\hat{\chi}_-$ is trace-free part of ${\chi}_-.$ Putting \eqref{t+} and \eqref{t-} in \eqref{H1}, we have 
\begin{equation*}
-\frac{1}{2\theta_-}\delta_{X}|\mathbf{H}|^2=-\Delta\varphi+2\langle W,\nabla \varphi\rangle +\left(\text{div}_{\Sigma}W-|W|^2+\overline{Q}\right)\varphi,
\end{equation*}
where  $\overline{Q}=\frac{1}{2}R_{\Sigma}-\frac{1}{2}G(l_+,l_-) +\tfrac{\theta_+}{2\theta_-}\left( |\hat{\chi}_-|^2+ G(l_-,l_-)\right) +\frac{3}{4}\theta_-\theta_+$.
\end{proof}

\section{Estimates for $\mathbf{H}$-stable surfaces}
As mentioned in the Introduction, it is no longer clear which surfaces are to be estimated in the more general context of spacetimes and initial data sets. Here we introduce a new notion of stability for surfaces in spacetime.
\begin{definition}\label{DefstableH}
	A two-sided spacelike surface $\Sigma$ embedded in an initial data set $(M,g,k)$ is \emph{$\mathbf{H}$-stable} with respect to direction $V$ if and only if there exists a smooth function $\psi\geq 0$ and $\psi\not\equiv 0$ such that $\delta_{\psi V}|\mathbf{H}|^2\geq 0$. Moreover, it is called \emph{strictly $\mathbf{H}$-stable} with respect to the
	direction $V$ if, moreover, $\delta_{\psi V}|\mathbf{H}|^2\neq 0$ somewhere on $\Sigma$.
\end{definition}
First, consider a $\mathbf{H}$-stable with respect to direction $N$, volume preserving surface $\Sigma$ with spacelike mean curvature vector (i.e., $H>|P|$). The $\mathbf{H}$-stable volume preserving condition implies that there exists smooth function $\varphi$ such that $\delta_{\varphi V}|\mathbf{H}|^2\geq 0$ and
\begin{equation}
\int_{\Sigma}\varphi d\mu_{\Sigma}=0.
\end{equation}Then, we now observe the following estimate.
\begin{theorem}\label{thmm1}
Let $\Sigma$ be a $\mathbf{H}$-stable with respect to direction $N$, volume preserving, topologically $\mathbb{S}^2$, surface embedded in an initial data set $(M,g,k)$ with unit normal $N$ and spacelike mean curvature vector. Then
\begin{equation}
1 +\frac{1}{24\pi} \int_{\Sigma} \theta_+\theta_- d\mu_{\Sigma}  \geq \frac{1}{12\pi} \int_{\Sigma} \left(\mu+J(N) - \theta_+k(N,N)-2\frac{\theta_+}{H}\nabla_N P\right) d\mu_{\Sigma}
\end{equation}
with equality if and only if both $k_{\Sigma}+A=0$ and $W=0$. 
\end{theorem}
\begin{proof}By Definition \ref{DefstableH} and Lemma \ref{variationlemma}, $\mathbf{H}$-stable with respect to direction $N$ condition implies that
\begin{equation*}
	\begin{split}
0 \leq \frac{1}{2H}\delta_{\varphi V}|\mathbf{H}|^2=&-\Delta\varphi+\frac{1}{2}\left[R_{\Sigma}-2\mu-(\text{tr}_gk)^2 +|k|^2-|A|^2-H^2 \right] \varphi\\ 
	&-2\frac{P}{H} \langle W,\nabla \varphi\rangle-\frac{P}{H} \left[-J(N)+\text{div}\,W+Hk(N,N)-\langle A,k_{\Sigma}\rangle \right] \varphi.
	\end{split} 
	\end{equation*}Add and subtract $\delta_{\varphi V}P$ and using equation \eqref{vart+}, we have
\begin{equation}
\begin{split}
0\leq &-\Delta \varphi+2\langle W,\nabla \varphi\rangle +\left(Q+ \text{div}\,W-|W|^2+\theta_+\text{tr}_gk -\frac{1}{2}\theta_+^2\right)\varphi\\
&-2\left(1+\frac{P}{H}\right) \langle W,\nabla \varphi\rangle-\frac{\theta_+}{H} \left[-J(N)+\text{div}\,W+Hk(N,N)-\langle A,k_{\Sigma}\rangle \right] \varphi
\end{split}
\end{equation}Multiply by $\varphi$ and integrate by parts yields to 
\begin{equation}\label{v2}
\begin{split}
0&\leq \int_{\Sigma}\left(|\nabla \varphi|^2 +\left(Q+ \text{div}\,W-|W|^2+\theta_+\text{tr}_gk -\frac{1}{2}\theta_+^2\right)\varphi^2\right)d\mu_{\Sigma}+\int_{\Sigma}\text{div}\left(\frac{P}{H}W\right)\varphi^2d\mu_{\Sigma}\\
&-\int_{\Sigma}\frac{\theta_+}{H} \left[-J(N)+\text{div}W+Hk(N,N)-\langle A,k_{\Sigma}\rangle \right] \varphi^2d\mu_{\Sigma}
\end{split}
\end{equation}Combing this with \cite[see page 242]{DanLee}
\begin{equation}
\begin{split}
\text{div}W+Hk(N,N)-\langle A,k_{\Sigma}\rangle&=\text{div}_gk(N)-(\nabla_Nk)(N,N)\\
&=\text{div}_gk(N)-\nabla_N(k(N,N))-2\langle W,\nabla\log\varphi\rangle\\
&=\text{div}_gk(N)-\nabla_N\text{tr}_{g}k-\nabla_NP-2\langle W,\nabla\log\varphi\rangle\\
&=J(N)-\nabla_NP-2\langle W,\nabla\log\varphi\rangle
\end{split}
\end{equation}we have 
\begin{equation}\label{v21}
\begin{split}
0&\leq \int_{\Sigma}\left(|\nabla \varphi|^2 +\left(Q+ \text{div}\,W-|W|^2+\theta_+\text{tr}_gk -\frac{1}{2}\theta_+^2\right)\varphi^2\right)d\mu_{\Sigma}+\int_{\Sigma}\text{div}\left(\frac{P}{H}W\right)\varphi^2d\mu_{\Sigma}\\
&+\int_{\Sigma}\frac{\theta_+}{H} \left[\nabla_NP+2\langle W,\nabla\log\varphi\rangle \right] \varphi^2d\mu_{\Sigma}
\end{split}
\end{equation}By integrating by parts the last term in second line, we obtain 
\begin{equation}\label{v22}
\begin{split}
0&\leq \int_{\Sigma}\left(|\nabla \varphi|^2 +\left(Q+ \text{div}\,W-|W|^2+\theta_+\text{tr}_gk -\frac{1}{2}\theta_+^2\right)\varphi^2\right)d\mu_{\Sigma}+\int_{\Sigma}\text{div}W\varphi^2d\mu_{\Sigma}\\
&+\int_{\Sigma}\frac{\theta_+}{H}\nabla_NP \varphi^2d\mu_{\Sigma}
\end{split}
\end{equation}By Li and Yau \cite{LY82}, there exists conformal map $\Psi=(\Psi^1,\Psi^2,\Psi^3):\Sigma \to S^2$ of degree $d$ with $\sum_i(\Psi^i)^2=1$, $\int_{\Sigma}|\nabla \Psi^i|^2=\frac{8\pi d}{3}$. Substituting $\varphi=\Psi^i$ into  \eqref{v22} and integrating by parts we have
\begin{equation}
\begin{split}
0\leq& 8\pi d +\int_{\Sigma}\left(Q-|W|^2+\theta_+\text{tr}_gk -\frac{1}{2}\theta_+^2\right)d\mu_{\Sigma}+\int_{\Sigma}\frac{\theta_+}{H}\nabla_NPd\mu_{\Sigma}
\end{split}
\end{equation}Setting $d=1$ and applying definition of $Q$ and the Gauss-Bonnet theorem yields to 
\begin{equation}\label{v23}
\begin{split}
12\pi\geq&  \int_{\Sigma}\left(\mu +J(N)+\frac{1}{2}|\chi_+|^2+|W|^2-\theta_+\text{tr}_gk +\frac{1}{2}\theta_+^2\right)d\mu_{\Sigma}-\int_{\Sigma}\frac{\theta_+}{H}\nabla_NPd\mu_{\Sigma}\\
\end{split}
\end{equation}Observe that
\begin{equation}
\begin{split}
\frac{1}{2}\theta_+^2-\theta_+\text{tr}_gk&=\frac{1}{2}\left(H^2+P^2+2HP\right)-HP-P^2-\theta_+k(N,N)\\
&=\frac{1}{2}\left(H^2-P^2\right)-\theta_+k(N,N)\\
&=-\frac{1}{2}\theta_+\theta_--\theta_+k(N,N)
\end{split}
\end{equation}Putting this in \eqref{v23}, we have the estimate.
\end{proof}
The Hawking energy for 2-surfaces embedded in a spacetime is 
\begin{equation}
E_{H}=\sqrt{\frac{|\Sigma|}{16\pi}}\left(1+\frac{1}{16\pi}\int_{\Sigma}\theta_+\theta_-d\mu_{\Sigma}\right).
\end{equation}Then we have the following result.
\begin{theorem}\label{thmm2}
	Let $\Sigma$ be a $\mathbf{H}$-stable with respect to direction $-l_-$, topologically $S^2$, embedded in spacetime $(L^4,h)$ satisfying the null energy condition, i.e., $Ric_{L}(v,v)\geq 0$ for every future-directed null vector field $v$. Moreover, assume $\Sigma$ has spacelike mean curvature vector $\mathbf{H}$ and variation speed is $\varphi\in C^{\infty}(\Sigma)$ such that $\int_{\Sigma}\varphi d\mu_{\Sigma}=0$. Then
	\begin{equation}
	E_{H}(\Sigma)  \geq \frac{\sqrt{|\Sigma|}}{48\pi^{3/2}}\int_{\Sigma} G(l_+,l_-) d\mu_{\Sigma},
	\end{equation}
	with equality if and only if $\hat{\chi}_-=W=Ric_L(l_-,l_-)=0$. 
\end{theorem}
\begin{proof}By Lemma \ref{variationlemma}, for $\mathbf{H}$-stable volume preserving surface with $X=-\varphi l_-$, we have $\delta_{X}|\mathbf{H}|^2\geq 0$. Combining this with $\theta_-<0$ implies that
	\begin{equation*}
		0 \leq -\frac{1}{2\theta_-}\delta_{X}|\mathbf{H}|^2=-\Delta\varphi+2\langle W,\nabla \varphi\rangle +\left(\text{div}_{\Sigma}W-|W|^2+\overline{Q}\right)\varphi,
	\end{equation*}
	where  $\overline{Q}=\frac{1}{2}R_{\Sigma}-\frac{1}{2}G(l_+,l_-) +\tfrac{\theta_+}{2\theta_-}\left( |\hat{\chi}_-|^2+ G(l_-,l_-)\right) +\frac{3}{4}\theta_-\theta_+$. Multiply by $\varphi$ and integrate by parts yields to
	\begin{equation}\label{v3}
	\begin{split}
	0\leq \int_{\Sigma}\left(|\nabla \varphi|^2+\left(\overline{Q}-|W|^2\right)\varphi^2\right)d\mu_{\Sigma}
	\end{split}
	\end{equation}Similar to proof of Theorem \ref{thmm1}, by Li and Yau \cite{LY82}, there exists a conformal map $\Psi=(\Psi^1,\Psi^2,\Psi^3):\Sigma \to S^2$ of degree $d$ with $\sum_i(\Psi^i)^2=1$, $\int_{\Sigma}|\nabla \Psi^i|^2=\frac{8\pi d}{3}$. Substituting $\varphi=\Psi^i$ into  \eqref{v3} and definition of $\overline{Q}$ we have
	\begin{equation}
	\begin{split}
	0\leq& 8\pi d +\int_{\Sigma}\left(\overline{Q}-|W|^2\right)d\mu_{\Sigma}\\
	&=8d\pi+\int_{\Sigma}\left(\frac{1}{2}R_{\Sigma}-\frac{1}{2}G(l_+,l_-) +\tfrac{\theta_+}{2\theta_-}\left( |\hat{\chi}_-|^2+ G(l_-,l_-)\right) +\frac{3}{4}\theta_-\theta_+-|W|^2\right)d\mu_{\Sigma}
	\end{split}
	\end{equation}Setting $d=1$ and using the Gauss-Bonnet theorem, $\tfrac{\theta_+}{\theta_-}<0$, and the null energy condition, we obtain 
	\begin{equation}
	\begin{split}
	12\pi+\frac{3}{4}\int_{\Sigma}\theta_-\theta_+d\mu_{\Sigma}\geq&  \int_{\Sigma}\left(G(l_+,l_-) -\tfrac{\theta_+}{2\theta_-}\left( |\hat{\chi}_-|^2+ G(l_-,l_-)\right) +|W|^2\right)d\mu_{\Sigma}\\
	&\geq \int_{\Sigma}G(l_+,l_-) d\mu_{\Sigma}.
	\end{split}
	\end{equation}The result follows from this inequality and definition of the Hawking energy.
	
\end{proof}

\section{Global Results for Stable Surfaces}
First we recall the definition of stable MOTS from \cite[Definition 2]{AnderssonMarsSimon0}.
\begin{definition}\label{Defstable}
	A MOTS $\Sigma$ embedded in an initial data set $(M,g,k)$ is \emph{stably outermost or stable} with respect to direction $V$ if and only if there exists a smooth function $\psi\geq 0$ and $\psi\not\equiv 0$ such that $\delta_{\psi V}\theta_+\geq 0$. Moreover, it is called \emph{strictly stable outermost} with respect to the
	direction $V$ if, moreover, $\delta_{\psi V}\theta_+\neq 0$ somewhere on $\Sigma$.
\end{definition}
Note that for $\varphi\in C^{\infty}(\Sigma)$, the stability operator of MOTS is 
\begin{equation}
L(\varphi)=-\Delta \varphi+2\langle W,\nabla \varphi\rangle +\left(Q+ \text{div}\,W-|W|^2\right)\varphi,
\end{equation}
where
\begin{equation}
Q=\frac{1}{2}R_\Sigma - \mu - \langle J,N \rangle -\frac{1}{2}|\chi_+|^2.
\end{equation} 
As shown in Lemma 2 of \cite{AnderssonMarsSimon0}, $\Sigma$ is stable if and only if $\lambda_{1}(L)\geq 0$. Moreover, there exist a positive eigenfunction $u$ such that $L(u)=\lambda_1u$. Now, we prove the following result which is analogous to Fischer-Colbrie and Schoen \cite[Theorem 3]{FS80} with Cohn-Vossen type inequality of Wang \cite[Theorem 5.10]{W19}. 
\begin{theorem}\label{thm4.2}
Let $(M,g,k)$ be an oriented initial data set satisfying $\mu-|J|> 0$. Let $\Sigma$ be a complete non-compact stable MOTS with unit normal $N$ embedded in $(M,g,k)$. Then $\Sigma$ is diffeomorphic to $\mathbb{R}^2$, properly embedded, and satisfies the following inequality
\begin{equation}\label{43}
\int_{\Sigma} \left(\mu+J(N)\right)d\mu_{\Sigma} \leq 2\pi.
\end{equation}
\end{theorem}
\begin{proof}First, an argument similar to Theorem 8.8 in \cite{GL83} shows that $\Sigma$ cannot be a non-compact surface with finite area. In particular, we have 
\begin{proposition}
Let $\Sigma$ be a stable MOTS of finite area embedded in an initial data set $(M,g,k)$ satisfying $\mu-|J|>0$. Then $\Sigma$ is homeomorphic to $\mathbb{S}^2$. 
\end{proposition}
\begin{proof}For any smooth function $\varphi\in C_c^\infty(\Sigma)$, multiply the stability condition $\delta_{\varphi N}\theta_+\geq 0$ in Lemma \ref{variationlemma} by $\varphi$ and integrate by parts to get
\begin{equation}
\int_{\Sigma}\left(|\nabla \varphi|^2+(K_{\Sigma}-\mu-J(N)-\frac{1}{2}|\chi_+|^2)\varphi^2\right) d\mu_{\Sigma}\geq 0,
\end{equation}
where $K_{\Sigma}$ is the Gauss curvature of $\Sigma$. Since $\mu-|J|>0$ and $J(N)\geq -|J|$, we have $\mu+J(N)> 0$ that implies 
\begin{equation}
\int_{\Sigma}\left( |\nabla \varphi|^2+K_{\Sigma}\varphi^2\right) d\mu_{\Sigma}>0.
\end{equation}
If $\Sigma$ is compact just choose $\varphi=1$. Then by Gauss-Bonnet theorem we have Euler number of $\Sigma$ is positive or equivalently $\Sigma$ is diffeomorphic to $\mathbb{S}^2$. If $\Sigma$ is not compact, it follows from Theorem 8.11 of \cite{GL83} that $|\Sigma|=\infty$. In particular, Theorem 8.11 of \cite{GL83} states that given a constant $\alpha>\frac{1}{2}$ and a complete surface $\Sigma$, if $\int_{\Sigma}\left(|\nabla \varphi|^2+\alpha K_{\Sigma}\varphi^2\right) d\mu_{\Sigma}\geq 0$ for all $\varphi\in C_c^\infty(\Sigma)$, then $|\Sigma|=\infty$. This complete the proof.
\end{proof}
Therefore, $\Sigma$ has infinite area. Since $\Sigma$ is stable MOTS, its first eigenvalue is non-negative, i.e., $\lambda_1(L)\geq 0$, that implies the first eigenvalue of symmetrized stability operator $L_s=-\Delta+Q$ is non-negative, see computation in the proof of Theorem 2.1 of \cite{GallowaySchoen} which is similar to our Lemma \ref{Stablelemma} without boundary term. We claim $(\Sigma,g_{\Sigma})$ is conformal to either $\mathbb{C}$ or cylinder $\mathbb{R}\times S^1$. If this is not true, the universal cover of $\Sigma$ is a disk. Then by Theorem 1 of \cite{FS80} and stability of MOTS, there exists a positive function $u$ such that $L_s(u)=0$. By lifting the metric $g_{\Sigma}$ to universal cover, we have a positive solution $v$, $L_s(v)=0$ on universal cover. But by Corollary 3 of \cite{FS80}, since $P:=\mu+J(N)+\frac{1}{2}|\chi_+|^2\geq 0$, the operator $0=-\Delta v+aK_{\Sigma}v-Pv$, for $a\geq 1$, does not have positive solution. This implies $L_s(v)=0$ on universal cover cannot have a positive solution which is a contradiction with the fact that the universal cover of $\Sigma$ is disk. This complete the proof of the claim.
   
Define the $\tilde{g}_{\Sigma}=u^2g_{\Sigma}$ on $\Sigma$, where $u$ satisfies $L_s(u)=0$. Let $\tilde{K}_{\Sigma}$ be Gauss curvature of $\tilde{g}_{\Sigma}$. It follows from $L_s(u)=0$ that 
\begin{equation}\label{tildeK}
\begin{split}
u^2\tilde{K}_{\Sigma}&=K_{\Sigma}-u^{-1}\Delta_{g}u+|\nabla \log u|^2\\
&=\mu+J(N)+\frac{1}{2}|\chi_+|^2+ |\nabla \log u|^2\geq 0
\end{split}
\end{equation}By integration and a classic result of Cohn-Vossen \cite{CV35}, we have
\begin{equation}\label{44}
0\leq \int_{\Sigma} \tilde{K} d\mu_{\tilde{\Sigma}}\leq 2\pi \chi(\Sigma)
\end{equation}where $d\mu_{\tilde{\Sigma}}$ is volume form with respect to metric $\tilde{g}_{\Sigma}$. Now if $\Sigma$ is conformal to cylinder $\mathbb{R}\times S^1$, then since $\chi(\Sigma)=0$ we have $\tilde{K}_{\Sigma}=0$. Therefore, it follows from \eqref{tildeK} that $\mu+J(N)+\frac{1}{2}|\chi_+|^2=0$. This is contradicting the energy condition $\mu-|J|>0$. Let $\Sigma$ to be diffeomorphic to $\mathbb{R}^2$. Since $L_s(u)=0$, we have $\int_{B_{\Sigma}(p,R)}\frac{1}{u}L_s(u) d\mu_{{\Sigma}}=0$ where $B_{\Sigma}(0,R)$ is a geodesic ball in $(\Sigma,g_{\Sigma})$ centered at $p\in \Sigma$ of radius $R$. This leads to 
\begin{equation}\label{45}
\begin{split}
\int_{B_{\Sigma}(p,R)} \left(\mu +J(N)+\frac{1}{2}|\chi_+|^2\right)d\mu_{\Sigma} =\int_{B_{\Sigma}(p,R)} \left(K_{\Sigma}-\frac{1}{u}\Delta_{\Sigma}u\right)d\mu_{\Sigma} \\
\leq \int_{B_{\Sigma}(p,R)} \tilde{K}_{\Sigma}d\mu_{\tilde{\Sigma}} \leq \int_{\Sigma} \tilde{K}_{\Sigma}d\mu_{\tilde{\Sigma}}
\end{split}
\end{equation} 
Taking $R\to \infty$, \eqref{44} and \eqref{45} yield \eqref{43}. To show that $\Sigma$ is proper, one proceeds as in \cite{W19}, but with \eqref{43} instead of that obtained in Theorem 5.10 of \cite{W19}. 
\end{proof}
Next, we recall a distance and area growth estimates from Meeks-Perez-Ros \cite{MPR08} based on Theorem 1 of \cite{FS80}. First, we have the following definition.
\begin{definition}
Given a complete metric on a surface $\Sigma$ and a point $p\in \Sigma$, define the distance of $p$ to the boundary of $\Sigma$, $\text{dist}(p,\partial \Sigma)$, as the infimum of the lengths of all divergent curves in $\Sigma$ starting at $p$. A ray is a divergent minimizing geodesic in $\Sigma$. 
\end{definition}
It can be shown that if $\Sigma$ is not compact but $\partial \Sigma$ is compact, then there exists a ray starting at some point $p\in \partial \Sigma$.
\begin{theorem}\label{thm4.7}\cite[Theorem 2.8]{MPR08}
Let $\Sigma$ be a surface and suppose that there exist constants $a>\frac{1}{4}$ and $c>0$ such that the operator $-\Delta+aK_{\Sigma}-c$ is non-negative, i.e., it has principal eigenvalue $\lambda_1(-\Delta+aK_{\Sigma}-c)\geq0$. Then the distance from every point $p\in \Sigma$ to the boundary of $\partial\Sigma$ satisfies 
\begin{equation}
\text{dist}(p,\partial\Sigma)\leq \pi \sqrt{\left(1+\frac{1}{4a-1}\right)\frac{a}{c}}
\end{equation}
If $\Sigma$ is complete then it must be compact and $\chi(\Sigma)>0$.
\end{theorem}
\begin{theorem}\label{thm4.8}\cite[Theorem 2.9]{MPR08}
Let $\Sigma$ be a Riemannian surface, $x_0\in \Sigma$ and constants $0<R'<R<\text{dist}(x_0,\partial\Sigma)$. Suppose that for some $a\in (\frac{1}{4},\infty)$ and $q\in C^\infty(\Sigma)$, $q\geq 0$, the operator $-\Delta+aK_{\Sigma}-q$ is non-negative on $\Sigma$. Then 
\begin{equation}
\frac{8a^2}{4a-1}\frac{|B(x_0,R')|}{R^2}+\left( 1-\frac{R'}{R}\right)^2 \int_{B(x_0,R')} q\leq 2\pi a \left(1-\frac{R'}{R}\right)^{\frac{2}{1-4a}}
\end{equation}
and if $\Sigma$ is complete then $R\to |B(x_0,R)|$ grows at most quadratically, $q\in L^1(\Sigma)$ and the universal cover of $\Sigma$ is conformally $\mathbb{C}$. 
\end{theorem}

We now apply these theorems to obtain various global results for general surfaces with $\theta_+$ and $\theta_-$. First, for any two sided surface with a well defined pair of null normals $l_+$ and $l_-$ embedded in a spacetime with Lorentzian $h$,  we define a scalar quantity
\begin{equation}
\mathcal{G}(\Sigma):=-\frac{3}{4}\theta_+\theta_-+\frac{1}{2}G(l_+,l_-)-\frac{\theta_+}{2\theta_-}G(l_-,l_-)
\end{equation}
\begin{definition}
Let $(M,g,k)$ be an initial data set. A spacelike $2$-surface $\Sigma \subset M$ is said to be $2$-immersed if there exists a $2:1$ cover $p:\tilde{M}\to M$ of $M$ such that $\tilde{\Sigma}=p^{-1}(\Sigma)$ is embedded and two-sided in $\tilde{M}$. 
\end{definition}
In this definition, $(\tilde{M},p^*g,p^*k)$ is itself embedded in a $4$-dimensional spacetime, which we denote with a slight abuse of notation $(\tilde{L},p^*h)$, and we write $\tilde{G}\equiv \text{Ric}_{p^*h}-\frac{1}{2} R_{p^*h}{p^*h}$ for the Einstein tensor of $(\tilde{L},p^*h)$. We now observe the following.  
\begin{theorem}\label{thm4.81}
Let $(\Sigma,g_{\Sigma})$ be a complete spacelike $2$-surface $2$-immersed in a $3$-dimensional initial data set $(M,g,k)$. In the notation above, write $\tilde{\Sigma}= p^{-1}(\Sigma)$ for the two-sided embedded surface covering $\Sigma$. Suppose that $\tilde{\Sigma}$ is $\mathbf{H}$-stable with respect to direction $-l_-$ and has spacelike mean curvature vector. Let $\bar{\Sigma}$ be the universal cover of $\Sigma$. Then 
\begin{enumerate}
\item{If ${\mathcal{G}}(\tilde{\Sigma})\geq c>0$ on $\tilde{\Sigma}$, then $\Sigma$ is topologically $S^2$ or $\mathbb{R}\mathbb{P}^2$.}
\item{If ${\mathcal{G}}(\tilde{\Sigma})\geq 0$ on $\tilde{\Sigma}$, then:}
\begin{itemize}
\item[(i)] $\bar{\Sigma}$ has at most quadratic growth.
\item[(ii)] $\int_{\bar{\Sigma}}{\mathcal{G}}(\bar{\Sigma})$ and $\int_{\bar{\Sigma}}\frac{\theta^+}{\theta_-}|\hat{\chi}_-|^2$ are both finite\footnote{It is implicit here that the integrands be defined on $\bar{\Sigma}$}.
\item[(iii)] If $\Sigma$ has infinite fundamental group, then $\hat{\chi}_-=0$ and ${\mathcal{G}}(\tilde{\Sigma})=0$, so in particular this case does not occur if the dominant energy condition holds. In this case, $\Sigma$ has at most linear area growth and is diffeomorphic to a cylinder, a Mobius strip, a torus, or a Klein bottle. 
\end{itemize}
\end{enumerate}
\end{theorem}
\begin{proof}
Since $\Sigma$ is $2$-immersed, we pass to its two-sided cover $\tilde{\Sigma}=p^{-1}(\Sigma)$, which is embedded in $(\tilde{M},p^{*}g,p^{*}k)$ and for which we have a well defined pair of null normals $l_+$ and $l_-$. The surface $\tilde{\Sigma}$ is $\mathbf{H}$-stable with respect to direction $-l_-$. Then for any smooth function $\varphi\in C^\infty(\Sigma)$, multiply the stability condition $\delta_{-\varphi l_-}|\mathbf{H}|^2\geq 0$ in Lemma \ref{variationlemma} by $\varphi$ and integrate by parts, we obtain 
\begin{equation}
\int_{\tilde{\Sigma}} \left( |\nabla \varphi|^2+Q'\varphi^2 \right) d\mu_{\tilde{\Sigma}} \geq 0 
\end{equation}
where $Q'=K_{\tilde{\Sigma}}+\frac{\theta_+}{2\theta_-}|\hat{\chi}_-|^2-{\mathcal{G}}(\tilde{\Sigma})$. This implies that the following operator is non-negative 
\begin{equation}\label{tL}
\tilde{L}:=-\Delta+K_{\tilde{\Sigma}}+\frac{\theta_+}{2\theta_-}|\hat{\chi}_-|^2-\mathcal{G}(\tilde{\Sigma})
\end{equation} 
on $\tilde{\Sigma}$. The theorem is now proved by applying Theorem \ref{thm4.7} and Theorem \ref{thm4.8} to various forms for $\tilde{L}$. \\ \indent

\noindent(1) We note that the operator $\tilde{L}$ is related to operator in Theorem \ref{thm4.7} by setting $a=1$ and $-\frac{\theta_+}{\theta_-}|\hat{\chi}_-|^2+{\mathcal{G}}(\tilde{\Sigma})\geq {\mathcal{G}}(\tilde{\Sigma})\geq c>0$. Then Theorem \ref{thm4.7} implies the distance estimate $\text{dist}(p,\partial \tilde{\Sigma})\leq \frac{2\pi}{\sqrt{3c}}$. Since $\tilde{\Sigma}$ is complete, then Theorem \ref{thm4.7} implies that $\tilde{\Sigma}$ is compact with $\chi(\tilde{\Sigma})>0$.\\ \indent

\noindent(2) Since $\tilde{\Sigma}$ is stable, so is the universal cover $\bar{\Sigma}$. By Theorem \ref{thm4.8} it then follows that $\bar{\Sigma}$ has at most quadratic area growth and that ${\mathcal{G}}(\bar{\Sigma})$ is in $L^1(\bar{\Sigma})$. Since ${\mathcal{G}}(\tilde{\Sigma})$ and $-\frac{\theta_+}{\theta_-}|\hat{\chi}_-|^2$ are both non-negative (and thus on $\bar{\Sigma}$ as well), the conclusion follows. (iii) then follows from (i) and (ii).

\end{proof}

\section{Stable Capillary Marginally outer trapped surfaces}\label{Section5}
Let $(M^n,g,k)$ be an initial data set and $\Sigma^{n-1}\subset M^n$ be an embedded, connected, complete hypersurface in $M$ with boundary $\partial M$, such that $\partial \Sigma$ is an embedded submanifold of $\partial M$. Call any such $\Sigma$ a capillary surface.\\ \indent 
Suppose that $\Sigma$ separates $M^n$ into two components, label one of these $\Omega$, and let the unit normal $N$ of $\Sigma$ point into $\Omega$. Let $\overline{\nu}$ be the unit normal of $\partial \Sigma$, pointing out of $\Omega$, as a subset of $\partial M$, and let $\gamma$ be the contact angle between $\partial M$ and $\Sigma$. Moreover, let $\nu$ be the unit outward normal vector of $\partial\Sigma$ and $\bar{N}$ be the unit outward normal vector of $\partial M$. Clearly, the relation of normal vectors is $\nu=\cos\gamma \bar{\nu}+\sin\gamma\bar{N}$ and $\bar{N}=\cos\gamma N+\sin\gamma\nu$ and if $\gamma=\frac{\pi}{2}$, $\Sigma$ is called a free boundary surface with $\nu=\bar{N}$, see Figure \ref{fig1}. In addition, we denote the second fundamental form and mean curvature of $\Sigma$ in the direction pointing towards $\Omega$ by $A$ and $H$. We let $\Pi$ and $H_{\partial M}$ be the second fundamental form and mean curvature of $\partial M$ in $M$. \\ \indent 
Now assume $l_+=\tau+N$ is the outward future-directed unit normal vector of $\Sigma$ in ambient spacetime with spacelike hypersurface $M^n$. As in Section \ref{Sec2}, denote $\chi_+$  and $\theta_+$ the outer null second fundamental form pointing towards $\Omega$ and the associated outer null expansion.\\

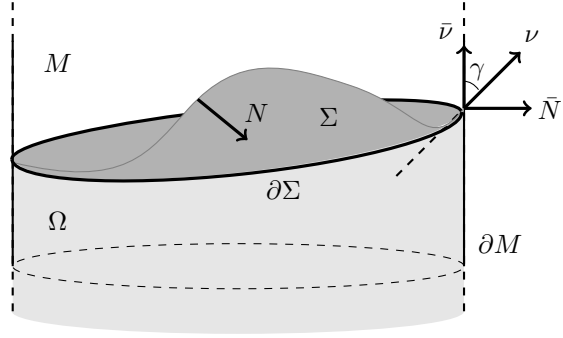
\begin{figure}[h]
	\centering
	\begin{tikzpicture}[scale=.6, every node/.style={scale=1}]
	\fill[black!30!white,rotate=6] (-1,.5) ellipse (5cm and .6cm);
	\fill[gray!20!white](-6,-3.5)--(-6,-.2).. controls (-3.5,-.8)and(-3,0) .. (-2.1,1).. controls (-1,2.2)and(0,2)..(1.5,1.5).. controls (3,1)and(3.2,0)..(4,1)--(4,-3.5) arc(360:180:5.01cm and  .5cm)--(-6,-3.5);
	\draw[dashed,black] (-1,-2.5) ellipse (5cm and .5cm);;
	\draw[thick,black](-6,-2.5)--(-6,2.5);
	\draw[thick,black](4,-2.5)--(4,2.5);
	\draw[very thick,->](4,1)--(5.25,2.25);
	\draw[dashed,thick](2.5,-.5)--(4,1);
	\draw[very thick,->](4,1)--(4,2.4);
	\draw[very thick,->](-1,.92)--(-1.7,-.2);
	\draw[very thick,black,rotate=6] (-6,0.5) arc(180:0:5.01cm and  .6cm);
	\fill[black!30!white](4,1) {[rotate=6.4] arc(360:180:5.01cm and  .86cm)}-- (-6,-.2).. controls (-3.5,-.8)and(-3,0) .. (-2.1,1).. controls (-1,2.2)and(0,2)..(1.5,1.5).. controls (3,1)and(3.2,0)..(4,1);
	\draw[gray](-6,-.2).. controls (-3.5,-.8)and(-3,0) .. (-2.1,1).. controls (-1,2.2)and(0,2)..(1.5,1.5).. controls (3,1)and(3.2,0)..(4,1);
	\draw[very thick,->](4,1)--(5.5,1);
	\node at  (-5,-1.5){$\Omega$};
	\node at (4.8,-2){$\partial M$};
	\node at (5.5,2.6){$\nu$};
	\node at (3.6,2.7){$\bar{\nu}$};
	\node at (5.9,1){$\bar{N}$};
	\node at (1,0.8){$\Sigma$};
	\node at (4.3,1.8){$\gamma$};
	\node at (0,-.8){$\partial\Sigma$};
	\node at (-5,2){$M$};
	\draw[very thick,->](-1.9,1.2)--(-.8,0.3);
	\node at (-.6,.9){$N$};
	\draw[very thick,black,rotate=6] (-6,0.5) arc(180:360:5.01cm and  .86cm);
	\draw[black] (4,1.6) arc(90:50:.5cm and  .8cm);
	\draw[thick,dashed](-6,-3.5)--(-6,3.5);
	\draw[thick,dashed](4,-3.5)--(4,3.5);
	\end{tikzpicture}
	\caption{Initial data set with capillary surface $\Sigma$.}\label{fig1}
\end{figure}

Now consider the smooth embedding 
\begin{equation}
f:\Sigma\times [0,T)\to M,\qquad f(\Sigma,t):=f_t(\Sigma)=\Sigma_t\subset M,
\end{equation} such that $\Sigma=\Sigma_0$. A variational vector field is $X_t=\frac{\partial f_t}{\partial t}$ such that $X=X_0=\frac{\partial f_t}{\partial t}|_{t=0}$ to be tangent to $\partial M$. We define the following functional
\begin{equation}\label{capillaryfunctional}
F[\Sigma_t] \equiv \int_{\Sigma_t} \theta^+_t \langle X_t,N_t \rangle \:  d\mu_{\Sigma_t}  + \int_{\partial \Sigma_t} \langle X_t,\nu_t-\text{cos}\gamma \overline{\nu}_t \rangle  d\mu_{\partial \Sigma_t}.
\end{equation}
Note that this functional is almost identical to that obtained by computing the first volume preserving variation for capillary surfaces \cite{RS97}. An important difference is that volume preserving variations imply that the contact angle $\gamma$ is constant, whereas $\gamma$ in \eqref{capillaryfunctional} is \textit{a priori} non-constant. Restricting to constant $\gamma$, we compute 
\begin{equation}
\begin{split}
\delta_XF[\Sigma]&:=\frac{\partial}{\partial t}\bigg|_{t=0}F[\Sigma_t]\\
&= \frac{\partial}{\partial t}\bigg|_{t=0}\left[\int_{\Sigma_t }\theta^+_t \: \langle X_t,N_t \rangle \: d\mu_{\Sigma_t} \right] + \frac{\partial}{\partial t}\bigg|_{t=0}\left[\int_{\partial \Sigma_t} \langle X_t,\nu_t -\text{cos}\gamma \overline{\nu}_t\rangle d\mu_{\partial \Sigma_t} \right].
\end{split}
\end{equation}
The first term gives
\begin{equation}
\frac{\partial}{\partial t}\bigg|_{t=0}\left[\int_{\Sigma_t} \theta^+_t \langle X_t,N_t \rangle \:  d\mu_{\Sigma_t}\right] = \int_{\Sigma}\left(D_X\theta^+ \: \varphi\: +\theta^+D_X\langle X,N \rangle+ \langle X,N \rangle\theta^+\text{div}_{\Sigma}X \right)d\mu_{\Sigma}.
\end{equation}
Since $\theta^+=0$ for MOTS, the second and third terms disappear. Moreover splitting $X=\varphi N+\hat{X}$ into normal and tangential components, where $\varphi\in C^{\infty}(\Sigma)$ and $\hat{X}\in T\Sigma$, we observe that $D_{\hat{X}}(\theta_+)=0$ and by Lemma \ref{variationlemma} it yields
\begin{equation}\label{bulkvariation}
\frac{\partial}{\partial t}\bigg|_{t=0}\left[\int_{\Sigma_t} \theta^+_t \langle X_t,N_t \rangle \:  d\mu_{\Sigma_t}\right]=\int_\Sigma\left[ -\varphi \Delta \varphi+2\varphi \langle W,\nabla \varphi\rangle +(\text{div} W -|W|^2+Q)\varphi^2\right] d\mu_{\Sigma}.
\end{equation}As for the second term, we get
\begin{equation}
\begin{split}
\frac{\partial}{\partial t}\bigg|_{t=0}\left[\int_{\partial \Sigma_t} \langle X_t,\nu_t-\text{cos}\gamma \overline{\nu}_t \rangle d\mu_{\partial \Sigma_t} \right] &= \int_{\partial \Sigma}  \langle \nabla_X X,\nu-\text{cos}\gamma \overline{\nu} \rangle d\mu_{\partial \Sigma} \\
& +\int_{\partial \Sigma} \langle X, \nabla_X (\nu- \text{cos}\gamma \overline{\nu})\rangle d\mu_{\partial \Sigma} \\
&+ \int_{\partial \Sigma} \langle X,\nu-\text{cos}\gamma \overline{\nu} \rangle \: \text{div}_{\partial \Sigma} X  \: d\mu_{\partial \Sigma}.
\end{split}
\end{equation}Since $\langle X,\nu-\text{cos}\gamma \overline{\nu}\rangle =0$, only the first two terms need be computed. A computation in Appendix of \cite{RS97}\footnote{We have a sign difference with \cite{RS97} for definition of second fundamental form and mean curvature. In particular, if $w$ is a unit normal to a surface, they define second fundamental form as $-\langle \nabla_{\cdot}w,\cdot \rangle$. Moreover their second fundamental form $\Pi$ also use normal $-\bar{N}$.} shows
\begin{equation}\label{bdryvariation}
\langle \nabla_X X,\nu -\text{cos}\gamma \overline{\nu}\rangle +\langle X, \nabla_X (\nu-\text{cos}\gamma\overline{\nu})\rangle = \varphi \frac{\partial \varphi}{\partial \nu}-\left[-\text{cot}\gamma A(\nu,\nu)+\frac{1}{\text{sin}\gamma}\Pi(\overline{\nu},\overline{\nu}) \right]\varphi^2.
\end{equation}
Letting $q=-\text{cot}\gamma A(\nu,\nu)+\frac{1}{\text{sin}\gamma}\Pi(\overline{\nu},\overline{\nu})$ and combining \eqref{bulkvariation} and \eqref{bdryvariation}, we have the following bilinear form
\begin{equation}
\delta_XF[\Sigma]=\int_\Sigma \left[|\nabla\varphi|^2+2\varphi \langle W,\nabla \varphi\rangle +\left(\text{div} W -|W|^2+Q\right)\varphi^2\right]d\mu_{\Sigma}- \int_{\partial\Sigma}q\varphi^2d\mu_{\partial\Sigma},
\end{equation}for all $\varphi\in C^{\infty}(\Sigma)$. In contrast to the second variation of area functional, this bilinear form is not symmetric and we cannot use Rayleigh-Ritz formula to find principle eigenvalue and define stability. However, there is a variational representation of the eigenvalue by Donsker and Varadhan \cite{DV} which applies in this setting, see Section 4 of \cite{AnderssonMarsSimon}. In particular, the corresponding eigenvalue problem  to functional $\delta_XF[\Sigma]$ is
\begin{equation}\label{OperatorL}
L(\varphi):=-\Delta \varphi + 2\langle W,\nabla \varphi\rangle +(\text{div} W -|W|^2+Q)\varphi =\lambda \varphi,\qquad \text{on $\Sigma$},
\end{equation}
\begin{equation}\label{RobinB}
B(\varphi):=\frac{\partial \varphi}{\partial \nu}-q\varphi=0, \qquad \text{on $\partial \Sigma$}.
\end{equation}
We define stability of capillary MOTS using the variation of the future null expansion.
\begin{definition}\label{Defcap}
	A capillary marginally outer trapped surface (MOTS) $\Sigma\subset M$ with constant contact angle $\gamma$ is \emph{stable} with respect to variation vector field $X$ if and only if there exists a non-negative function $\varphi\in C^{\infty}(\Sigma)$, $\varphi\not\equiv 0$ satisfying Robin boundary condition $B(\varphi)=0$ such that $\delta_X\theta_+\geq 0$. Moreover, it is called \emph{strictly stably outermost} with respect to the
	direction $X$ if, moreover, $\delta_X\theta_+\neq 0$ somewhere on $\Sigma$.
\end{definition}
The operator $\mathcal{L}:=(L,B)$ is called stability operator with Robin boundary condition. By Theorem \ref{ThmA1}, if $q\leq 0$, There is a real eigenvalue $\lambda_1$, called the principal eigenvalue, such that for any other eigenvalue $\lambda'$, $\text{Re}\lambda'\geq \lambda_1$. The associated eigenfunction $\varphi$, $\mathcal{L}(\varphi)=(L(\varphi),B(\varphi)=(\lambda_1\varphi,0)$, is unique up to a multiplicative constant, and
can be chosen to be strictly positive and it yields to.

\begin{lemma}
	Let $\Sigma$ be a capillary MOTS $\Sigma\subset M$ with constant contact angle $\gamma$ and $q\leq 0$ on $\partial\Sigma$. Then $\Sigma$ is capillary stable MOTS if and only if the principal eigenvalue of stability operator $\mathcal{L}$ satisfies
	\begin{equation}
	\lambda_1(\mathcal{L}) \geq 0.
	\end{equation}Moreover, it is strictly stable if and only if $\lambda_1(\mathcal{L})>0$
\end{lemma}
\begin{proof}
	If $\lambda_1(\mathcal{L}) \geq 0$, then by Thoerem \ref{ThmA1}, there exist a positive function $\varphi$ such that $\mathcal{L}\varphi=(\lambda_1\varphi,0)$ which implies $\delta_X\theta\geq 0$ and $B(\varphi)=0$. Conversely, since $\Sigma$ is stable and $\lambda_1(\mathcal{L})$ is also eigenvalue of $\mathcal{L}^*$ with positive eigenfunction $\varphi^*$, we have $\lambda_1\langle \psi,\varphi^*\rangle=\langle \psi,L^*\varphi^*\rangle=\langle L\psi,\varphi^*\rangle\geq 0$ which implies $\lambda_1(\mathcal{L}) \geq 0$.
\end{proof}
The `symmetrized' version of Definition \ref{Defcap} is as follows.
\begin{definition}\label{Defcapsym}
	Let $\Sigma$ be a capillary surface with constant contact angle $\gamma$. The symmetric operator $\mathcal{L}_s=(L_s,B_s)$ on $\Sigma$ defined as
	\begin{equation}
	L_s(\varphi):=-\Delta_{\Sigma}\psi +Q\psi,\quad \text{on $\Sigma$}\qquad\qquad B_s(\varphi):=\frac{\partial \psi}{\partial \nu} -\left(q-\langle W,\nu\rangle\right)\psi \quad \text{on $\partial\Sigma$}
	\end{equation}Then $\Sigma$ is called \emph{capillary symmetric stable} if 
	\begin{equation}
	\lambda_1(\mathcal{L}_s)\geq 0
	\end{equation}
	where $\lambda_1$ is the principal eigenvalue of $\mathcal{L}_s$ for eigenfunctions satisfying the homogeneous Robin boundary condition.
\end{definition}
Note that $\mathcal{L}_s$ is symmetric and all eigenvalues are real, therefore, we do not need extra assumption on the sign of $q$ on $\partial\Sigma$. Therefore, we have the following result. 
\begin{lemma}\label{Stablelemma}Let $\Sigma$ be a capillary surface with constant contact angle $\gamma$ and $q\leq 0$ in \eqref{RobinB}. Then $\lambda_1(\mathcal{L}) \leq \lambda_1(\mathcal{L}_s)$ and stability of $\Sigma$, $\lambda_1(\mathcal{L})\geq 0$, yields to a positive-semidefinite bilinear form
	\begin{equation}
	P(u,u):=\int_{\Sigma} \left(|\nabla u|^2 +Qu^2\right)d\mu_{\Sigma}-\int_{\partial \Sigma} \left(q-\langle W,\nu\rangle\right) u^2 d\mu_{\partial\Sigma} \geq 0
	\end{equation} for all $u\in C^{\infty}(\Sigma)$. Moreover, if $P(1,1)=0$, then $Q=0$, $W=\nabla\log\varphi$, $q=\langle W,\nu\rangle$, and $\lambda_1(L_s)=0$.
\end{lemma}
\begin{proof}Let $\lambda_1(\mathcal{L})$ be the principle real eigenvalue of $\mathcal{L}$ with positive eigenfunction $\varphi$.
	By a computation in \cite{GallowaySchoen}
	\begin{equation}
	\mathcal{L}(\varphi) =\text{div}(W-\nabla \text{log}\varphi)\varphi -|W-\nabla \text{log}\varphi|^2\varphi +Q\varphi.
	\end{equation}Let $u \in C^{\infty}(\Sigma)$. Multiplying above equation by $u\varphi^{-1}$ and completing the square leads to 
	\begin{equation}
	u^2\varphi^{-1}\mathcal{L}(\varphi)=\text{div}(u^2(W-\nabla \text{log}\varphi)) +|\nabla u|^2+Qu^2-|(W-\nabla \text{log}\varphi)u+\nabla u|^2
	\end{equation}
	Integrating by parts and we have
	\begin{equation}\label{Lambdaineq}
	\begin{split}
	\lambda_1(\mathcal{L}) \int_{\Sigma} u^2 d\mu_{\Sigma}&=\int_{\Sigma} u^2\varphi^{-1}\mathcal{L} (\varphi) d\mu_{\Sigma}\\
	&=\int_{\Sigma} \left(\text{div}(u^2(W-\nabla \text{log}\varphi)+ |\nabla u|^2 +Qu^2 -|(W-\nabla \text{log}\varphi)u+\nabla u|^2\right)d\mu_{\Sigma}\\
	&\leq  \int_{\Sigma} \left(|\nabla u|^2 +Qu^2-|(W-\nabla \text{log}\varphi)u+\nabla u|^2\right)d\mu_{\Sigma}+\int_{\partial \Sigma} \langle u^2(W-\nabla \text{log}\varphi), \nu \rangle d\mu_{\partial\Sigma}  \\
	&\leq  \int_{\Sigma} \left(|\nabla u|^2 +Qu^2\right)d\mu_{\Sigma}+\int_{\partial \Sigma} \langle u^2(W-\nabla \text{log}\varphi), \nu \rangle d\mu_{\partial\Sigma} \\
	&=\int_{\Sigma} \left(|\nabla u|^2 +Qu^2\right)d\mu_{\Sigma}-\int_{\partial \Sigma} \left(q-\langle W,\nu\rangle\right) u^2 d\mu_{\partial\Sigma} 
	\end{split}
	\end{equation}where the last equality follows from Robin boundary condition \eqref{RobinB}. The result now follows from Rayleigh quotient characterization. If equality holds, then combining $P(1,1)=0$, $P(u,u)\geq 0$ for all $u\in C^{\infty}(\Sigma)$, and Cauchy-Schwarz inequality, we have $P(1,u)=0$ for all $u\in C^{\infty}(\Sigma)$. Now since $u$ is arbitrary smooth function, we have $Q=0$ and $q=\langle W,\nu\rangle$. Moreover, $P(1,1)=0$ together with the first inequality in \eqref{Lambdaineq}, we have $W=\nabla\log\varphi$. Moreover, $P(1,1)=0$ and $\lambda_1(\mathcal{L})=0$, yields to $\lambda_1(\mathcal{L}_s)=0$. This complete the proof.
\end{proof}

\section{Free Boundary Marginally Outer Trapped Surfaces}\label{Section6}
\subsection{Rigidity of Stable Free Boundary Marginally Outer Trapped Surfaces}
For free boundary surface, the contact angle in previous section is $\gamma=\pi/2$, therefore, $\nu=\overline{N}$ and $\bar{\nu}=-N$. For a properly embedded surface $(\Sigma,\partial \Sigma)\subset (M,\partial M)$, define the functional $I(\Sigma)$
\begin{equation}
I(\Sigma)\equiv |\Sigma|\inf_{\Sigma}\: (\mu+J(N))+|\partial \Sigma| \inf_{\partial \Sigma} \left(H_{\partial M} - \langle W,\nu\rangle\right) 
\end{equation} 
for a surface $\Sigma$ of area $|\Sigma|$ and boundary length $|\partial \Sigma|$. 
\begin{proposition}\label{prop1rigidity}
	Let $(M,g,k)$ be a connected initial data set with boundary $\partial M$ having second fundamental form $\Pi(\cdot,\cdot)$ and outwards pointing normal $\nu$. Let $\Sigma$ be a free boundary stable MOTS in $M$ with unit normal $N$ and $\Pi(N,N)\leq 0$ on $\partial\Sigma$. Suppose that $H_{\partial M}-\langle W,\nu \rangle$ is bounded from below on $\partial\Sigma$ and $\mu+J(N)$ is bounded from below on $\Sigma$. Then 
	\begin{equation}\label{Iin}
	I(\Sigma)  \leq 2\pi \chi(\Sigma)
	\end{equation} 
	and equality holds if and only if $\Sigma$ satisfies the following properties.
	\begin{enumerate}
		\item $\chi_+=0$, $Q=0$, $W=\nabla\log\varphi$, and $q=\Pi(N,N)=\langle W,\nu \rangle$, 
		\item $\mu+J(N)$ is constant on $\Sigma$ and equal to $\inf_{\Sigma}(\mu+J(N))$,
		\item Geodesic curvature of $\partial\Sigma$ in $\Sigma$ is constant on $\partial \Sigma$ and equal to $\inf_{\partial \Sigma} \left(H_{\partial M} - \langle W,\nu\rangle\right)$,
		\item $\lambda_1(\mathcal{L}_s)=\lambda_1(\mathcal{L})=0$. 
		\setcounter{nameOfYourChoice}{\value{enumi}}
	\end{enumerate}Furthermore, if $H_{\partial M}$ and $-\langle W,\nu \rangle$ are bounded from below on $\partial\Sigma$ instead of $H_{\partial M}-\langle W,\nu \rangle$, inequality \eqref{Iin} holds and equality gives us additional properties.
	\begin{enumerate}
		\setcounter{enumi}{\value{nameOfYourChoice}}
		\item $\langle W,\nu \rangle$ is constant on $\Sigma$ and equal to $\sup_{\Sigma}\langle W,\nu \rangle\leq 0$, and
		\item $H_{\partial M}$ is constant on $\partial \Sigma$ and equal to $\inf_{\partial \Sigma} H_{\partial M}$.
	\end{enumerate}
\end{proposition} 
\begin{proof}
	Using the notation above, the free boundary condition implies that the geodesic curvature of $\partial \Sigma$ in $\Sigma$ can be computed as $\kappa=g(T,\nabla_T\nu)$ where $\nu$ points outside of $\partial M$ and coincides with $\overline{N}$ and $T$ is tangent vector to $\partial\Sigma$. Recalling that the mean curvature is the trace of the shape operator $\nabla \nu$, we have 
	\begin{equation}\label{HdM}
	H_{\partial M}=\kappa+g(N,\nabla_N \nu)=\kappa+q
	\end{equation}
	Since $\Sigma$ is a free boundary stable MOTS, by Lemma \ref{Stablelemma}, we have a positive semi-definite bilinear form
	\begin{equation}
	P(u,u) \equiv \int_{\Sigma} \left(|\nabla u |^2+Qu^2\right)d\mu_{\Sigma} - \int_{\partial \Sigma}(q-\langle W,\nu\rangle )u^2 d\mu_{\partial\Sigma}\geq 0
	\end{equation}for all $u\in C^{\infty}(\Sigma)$. Combining above equation with $u=1$, $Q=\frac{1}{2}R_{\Sigma}-\mu-J(N)-\frac{1}{2}|\chi^+|^2$ and \eqref{HdM}, we get 
	\begin{equation}\label{P(1,1)in}
	\begin{split}
	0&\geq \int_{\Sigma} \left(\mu+ J(N) +\frac{1}{2}|\chi_+|^2\right)d\mu_{\Sigma}-\frac{1}{2}\int_{\Sigma} R_{\Sigma}d\mu_{\Sigma}-\int_{\partial \Sigma}\kappa d\mu_{\partial\Sigma}+\int_{\partial \Sigma}\left(H_{\partial M}-\langle W,\nu \rangle\right) d\mu_{\partial\Sigma}\\
	&=\int_{\Sigma} \left(\mu+ J(N) +\frac{1}{2}|\chi_+|^2\right)d\mu_{\Sigma}+\int_{\partial \Sigma}\left(H_{\partial M}-\langle W,\nu \rangle\right) d\mu_{\partial\Sigma}-2\pi\chi(\partial\Sigma) \\
	&\geq |\Sigma| \inf_{\Sigma} (\mu+J(N)) +|\partial \Sigma|\inf_{\Sigma} \left(H_{\partial M}  - \langle W,\nu\rangle\right)- 2\pi \chi(\Sigma)\\
	&=I(\Sigma)-2\pi \chi(\Sigma)
	\end{split}
	\end{equation}where the first equality follows from the Gauss-Bonnet Theorem. Therefore, $I(\Sigma)\leq 2\pi \chi(\Sigma)$. The rigidity is equivalent to $\chi_+=0$, $\mu+J(N)$ is constant on $\Sigma$ and equal to $\inf_{\Sigma}(\mu+J(N))$, $H_{\partial M}-\langle W,\nu \rangle$ is constant on $\partial \Sigma$ and equal to $\inf_{\partial \Sigma} \left(H_{\partial M}-\langle W,\nu \rangle\right)$, and $P(1,1)=0$, thus, it follows from Lemma \ref{Stablelemma} that $Q=0$, $W=\nabla\log\varphi$, and $q=\Pi(N,N)=\langle W,\nu \rangle$. Moreover, since $\Pi(N,N)=\langle W,\nu \rangle$ and $\Pi(N,N)\leq 0$, we have $\sup_{\partial\Sigma} \langle W,\nu \rangle\leq 0$. Combining this with \eqref{HdM}, we get $\partial\Sigma$ has constant geodesic curvature $\kappa=\inf_{\Sigma}\left(H_{\partial M}-\langle W,\nu \rangle\right).$
	
	Now, if we assume $H_{\partial M}$ and $-\langle W,\nu \rangle$ are bounded from below on $\partial\Sigma$ instead of $H_{\partial M}-\langle W,\nu \rangle$, then from \eqref{P(1,1)in} we have $\langle W,\nu \rangle$ and $H_{\partial M}$ are constant on $\partial\Sigma$ and equal to $\sup_{\Sigma}\langle W,\nu \rangle\leq 0$ and $\inf_{\partial \Sigma} H_{\partial M}$, respectively.
\end{proof}
For rigidity of free boundary stable MOTS, we need to impose some assumption on neighborhood of embedding of $\Sigma$ in $M$. Given $\zeta>0$ and define $M_{\zeta}$ to be a neighborhood of $\Sigma$ with geodesic distance $\zeta$ and boundary $\partial M_{\zeta}$. Moreover, we define 
\begin{equation}
I_{\zeta}(\Sigma)\equiv |\Sigma|\inf_{M_{\zeta}}\: (\mu-|J|)+|\partial \Sigma| \inf_{\partial M_{\zeta}} \left(H_{\partial M_{\zeta}}  - \langle W,\nu\rangle\right).
\end{equation} Then we have the following rigidity result.
\begin{theorem}\label{prop2rigidity}Let $(M,g,k)$ be an initial data set with constants $c_i>0$ such that $\mu-|J|\geq c_1$ on $M_{\zeta}$, $H_{\partial M_{\zeta}}\geq -c_2$ on $\partial M_{\zeta}$, and $\langle W,\nu\rangle\leq c_3$ on $\partial M_{\zeta}$.  Let $\Sigma$ be a properly embedded, area minimizing, free boundary weakly outermost MOTS in $M_{\zeta}$ with $I_{\zeta}(\Sigma)=2\pi \chi(\Sigma)$ and $\max_{\partial\Sigma}\Pi(N,N)=0$. Finally assume that either one of the following holds.
	\begin{enumerate}[(i)]
		\item\label{i} $\partial \Sigma$ is locally length minimizing in $\partial M_{\zeta}$, or
		\item\label{ii} $\inf_{\partial M_{\zeta}} \left(H_{\partial M_{\zeta}}  - \langle W,\nu\rangle\right)=0$.
	\end{enumerate}
	Then
	\begin{enumerate}
		\item There exist $0<\epsilon<\zeta$ such that $M_{\zeta}$ splits as $([0,\epsilon)\times\Sigma, dt^2+g_\Sigma)$.
		\item $\Sigma_t=\{t\}\times\Sigma$, for all $t\in [0,\epsilon)$, is totally geodesic as a submanifold of spacetime, i.e., $\chi_+=\chi_-=0$, and as submanifold of $M$, i.e., $A=0$.
		\item $\Sigma_t$ has constant Gauss curvature $\inf_{M_{\zeta}}\mu$.
		\item $J=k(\cdot,\cdot)|_{T\Sigma_t}=k(\cdot,N_t)|_{T\Sigma_t}=Q_t=0$ and scalar curvature $R_M=2\mu\geq c_1$.
		\setcounter{nameOfYourChoice}{\value{enumi}}
	\end{enumerate}
	Moreover, 
	\begin{enumerate}
		\setcounter{enumi}{\value{nameOfYourChoice}}
		\item If we assume (\ref{i}), $\partial\Sigma_t$ in $\Sigma_t$  has constant geodesic curvature $\inf_{\partial M_{\zeta}}H_{\partial M_{\zeta}}$ and $\partial\Sigma_t$ in $\partial M_{\zeta}$  has vanishing geodesic curvature.
		\item  If we assume (\ref{ii}), $\partial\Sigma_t$ in $\Sigma_t$ has vanishing geodesic curvature and $H_{\partial M_{\zeta}}=0$.
	\end{enumerate}
\end{theorem}
\begin{remark}
	In \cite{A13}, there is mean convex  an assumption on boundary which is substitute with weaker assumptions $H_{\partial M_{\zeta}}\geq -c_2$ on $\partial M_{\zeta}$, $\langle W,\nu\rangle\leq c_3$ on $\partial M_{\zeta}$, and $\max_{\partial\Sigma}\Pi(N,N)=0$.
\end{remark}
\begin{proof}
	The argument combines \cite{A13} and \cite{GM18} and follows from three steps: (1) construct a foliation free boundary hypersurfaces $\{\Sigma_t\}$ each having constant $\theta_+(t)$, (2) show that $\theta_+(t)=0$ for each $\Sigma_t$, and (3) obtain the rigidity statements. \\ \\
	\noindent\emph{Step 1. Foliation $\{\Sigma_t\}$}. We seek to construct a foliation around $\Sigma$ of free boundary surfaces with constant $\theta_+$. For small $u\in C^{\infty}(\Sigma)$, we define the surface $\Sigma_u=\exp(u(x)N)$ and the future directed null expansion $\theta_+(u)$ of $\Sigma_u=\Sigma[u]$ such that $\Sigma_{0}=\Sigma$. Consider the following operator
	\begin{equation}
	\Psi: C^{2,\alpha}(\Sigma)\times \mathbb{R}\longrightarrow  C^{0,\alpha}(\Sigma)\times   C^{1,\alpha}(\partial\Sigma)\times \mathbb{R},
	\end{equation}such that 
	\begin{equation}
	\Psi(u,s)=\left(\theta_+(u)-k(s),g(N_u,\nu_u),\int_{\Sigma}u d\mu_{\Sigma}\right).
	\end{equation}Then we compute the linearization of $\Psi$ at $(0,0)$ as follows
	\begin{equation}
	\begin{split}
	D\Psi\big|_{(0,0)}(u,s)&=\frac{d}{dr}\big|_{r=0}\Psi(ru,rs)\\
	&=\left(L(u)-k'(0),-\frac{\partial u}{\partial\nu}+g(N,\nabla_N\nu)u,\int_{\Sigma}u d\mu_{\Sigma}\right),
	\end{split}
	\end{equation}where $L$ is stability operator of MOTS and $\frac{d}{dr}\big|_{r=0}g(N_{ru},\nu_{ru})=-\frac{\partial u}{\partial\nu}+g(N,\nabla_N\nu)u=-\frac{\partial u}{\partial\nu}+qu=-B(u)$, cf. Proposition 17 of \cite{A13}. To apply inverse function theorem we need to show $D\Psi\big|_{(0,0)}(u,s)$ is an isomorphism. In particular, we want to show that the operator $\mathcal{L}(u)=(f,\gamma)$ is invertible for $f\in C^{0,\alpha}(\Sigma)$ and $\gamma\in C^{1,\alpha}(\partial\Sigma)$. The corresponding eigenvalue problem is $\mathcal{L}(u)=(\lambda u,0)$. Since by Proposition \ref{prop1rigidity}, the principle simple eigenvalue $\lambda_1(\mathcal{L})=0$ with corresponding positive eigenfunction $\varphi$, the kernel of $\mathcal{L}$ is constant multiple of $\varphi$. By Fredholm alternative,  $\mathcal{L}(u)=(L(u), B(u))=(f,\gamma)$ has a unique solution over Banach spaces module kernel of $\mathcal{L}$ if and only if $\int_{\Sigma}f\varphi^*=0$, where $\varphi^*$ is eigenfunction of $\mathcal{L}^*$ correspond to eigenvalue $\lambda_1(\mathcal{L})=0$ (Note that by Theorem \ref{ThmA1}, $\lambda_1$ is also eigenvalue of the adjoint operator $\mathcal{L}^*$). Therefore, $D\Psi\big|_{(0,0)}(u,s)$ is invertible over Banach spaces module kernel of $\mathcal{L}$. Thus, by the inverse function theorem, there exist $\epsilon>0$, for $s\in (-\epsilon,\epsilon)$, and $u(s)\in C^{\infty}(\Sigma)$ such that $\epsilon<\zeta$ and $\Psi(u,s)=(0,0,s)$ or equivalently 
	\begin{equation}
	\theta_+(u(s))=k(s),\qquad g(N_u,\nu_u)=0,\qquad \int_{\Sigma}u(s) d\mu_{\Sigma}=s.
	\end{equation} Using chain rule and differentiate at $s=0$ yields to
	\begin{equation}\label{eqndiff}
	u'(0)L(0)=k'(0),\qquad -\frac{\partial u'(0)}{\partial\nu}+g(N,\nabla_N\nu)u'(0)=0,\qquad \int_{\Sigma}u'(0) d\mu_{\Sigma}=1.
	\end{equation}Since $\int_{\Sigma}k'(0)\varphi^*d\mu_{\Sigma}=0$ and $\varphi^*>0$, we have $k'(0)=0$. This shows that $u'(0)$ is in the kernel of $\mathcal{L}$. Combining this with integral condition in \eqref{eqndiff}, we have  $u'(0)=\text{constant}\cdot \varphi>0$. This means there exist a local coordinate $(t,x^i)$ such that for small $t$ surfaces $\Sigma_{t}:=\Sigma[u(t)]=\{t\}\times \Sigma$ form smooth
	foliation of constant $\theta_+(t)$ free boundary hypersurfaces.
	Let $h_t$ be induced metric on $\Sigma_t$ and $0<\varphi_t\in C^{\infty}(\Sigma)$, which is the speed of surfaces with variation vector $\frac{\partial}{\partial t}=\varphi_t N_t$ such that $\varphi_0=u'(0)$, then the metric on $\Sigma\times (-\epsilon,\epsilon)$ has the following form
	\begin{equation}
	g=\varphi_t^2dt^2+h_t.
	\end{equation}
	\noindent\emph{Step 2. Each $\Sigma_t$ is a stable free boundary MOTS.}
	We now show that $\theta_+(t)=0$ for all $t\in [0,\epsilon)$. First, observe that the weakly outermost (i.e., no outer trapped surface in $M_{\zeta}$ homologous to $\Sigma$) property of $\Sigma$ implies $\theta_+(t)\geq 0$ on $\Sigma_t$ for $t\in[0,\epsilon)$. Next, by Lemma \ref{Stablelemma} we have 
	\begin{equation}\label{vartheta+}
	\frac{d\theta_+(t)}{dt}=-\Delta_{t} \varphi_t+2\langle W_{t},\nabla_{t} \varphi_t\rangle +\left(Q_{t}+ \text{div}_{t}W_{t}-|W_{t}|^2+(\text{tr}_gk)_t\theta_+(t) -\frac{1}{2}\theta_+(t)^2\right)\varphi_t,
	\end{equation}and
	\begin{equation}\label{Bcfree}
	-\frac{\partial \varphi_t}{\partial\nu_t}+g(N_t,\nabla_{N_t}\nu_t)\varphi_t=0,
	\end{equation}
	where subscript $t$ means evaluated on $\Sigma_t$ and \eqref{Bcfree} follows from the variation of free boundary condition, cf. Proposition 18 of \cite{A13}. Since $\varphi_t>0$ we have
	\begin{equation}\label{theta'1}
	\begin{split}
	\frac{\theta'_{+}(t)}{\varphi_t}&=-\nabla_t\log\varphi_t +2\langle W_t,\nabla_t\log\varphi_t \rangle -|W_t|^2+\emph{div}_t W_t+Q_t+\theta_+(t) (\text{tr}_gk)_t-\frac{1}{2}\theta_+(t)^2\\
	&\leq \text{div}_t(W_t-\nabla_t \log \varphi_t)+Q_t+\theta_+(t) (\text{tr}_gk)_t.
	\end{split}
	\end{equation}Define the function 
	\begin{equation}
	\mathcal{J}(t):=\theta_+(t)e^{-\int_0^t\alpha(s)ds},\qquad\alpha(t):=\left(\int_{\Sigma_t}\frac{1}{\varphi_t}d\mu_{\Sigma_t}\right)^{-1}\int_{\Sigma_t}(\text{tr}_gk)_t d\mu_{\Sigma_t}
	\end{equation} such that $\mathcal{J}(0)=0$.
	Together with definition of $Q$, \eqref{Bcfree}, \eqref{theta'1}, and the fact that $\theta_+(t)$ is constant on $\Sigma_t$, we have 
	\begin{equation}\label{J'1}
	\begin{split}
	\beta(t)\mathcal{J}'(t)&=\theta'_{+}(t)\int_{\Sigma_t}\frac{1}{\varphi_t}d\mu_{\Sigma_t}-\theta_+(t)\int_{\Sigma_t}(\text{tr}_gk)_t d\mu_{\Sigma_t}\\
	&\leq \int_{\partial \Sigma_t} \langle W_t,\nu_t\rangle d\mu_{\partial \Sigma_t} -\int_{\partial \Sigma_t} \partial_{\nu_t}\log\varphi_t d\mu_{\partial \Sigma_t} +\int_{\Sigma_t} Q_t d\mu_{\Sigma_t}\\
	&= \int_{\partial \Sigma_t} \left(\langle W_t,\nu_t\rangle -\Pi(N_t,N_t)\right) d\mu_{\partial \Sigma_t}+ \int_{\Sigma_t} \left(K_t-\mu -J(N_t)-\frac{1}{2}|\chi_{t+}|^2\right) d\mu_{\Sigma_t}\\
	&\leq  \int_{\partial \Sigma_t} \left(\langle W_t,\nu_t\rangle -\Pi(N_t,N_t)\right) d\mu_{\partial \Sigma_t}+ \int_{\Sigma_t} \left(K_t-\mu +|J|-\frac{1}{2}|\chi_{t+}|^2\right) d\mu_{\Sigma_t},
	\end{split}
	\end{equation} where we used  $\partial_{\nu_t}\log\varphi_t=g(N_t,\nabla_{N_t}\nu_t)=q=\Pi(N_t,N_t)$ and $\beta(t)=\left(\int_{\Sigma_t}\frac{1}{\varphi_t}d\mu_{\Sigma_t}\right)e^{\int_0^t\alpha(s)ds}>0$. Combining this with the geodesic curvature equation \eqref{HdM} and Gauss-Bonnet theorem we get
	\begin{equation}\label{J'2}
	\begin{split}
	\beta(t)\mathcal{J}'(t)&\leq -|\partial {\Sigma_t}| \inf_{\partial M_{\zeta}}\left(H_{\partial M_{\zeta}}  - \langle W,\nu\rangle\right) -|\Sigma_t|\inf_{M_{\zeta}}\left(\mu -|J|\right) + 2\pi\chi(\Sigma_t) -\int_{\Sigma_t}\frac{1}{2}|\chi_{t+}|^2 d\mu_{\Sigma_t}\\
	&\leq -|\partial {\Sigma_t}| \inf_{\partial M_{\zeta}}\left(H_{\partial M_{\zeta}}  - \langle W,\nu\rangle\right) -|\Sigma_t|\inf_{M_{\zeta}}\left(\mu -|J|\right) + 2\pi\chi(\Sigma_t).
	\end{split}
	\end{equation}Together with the infinitesimal rigidity $I_{\zeta}(\Sigma)=2\pi \chi(\Sigma)$ and $\chi(\Sigma)=\chi(\Sigma_t)$ leads to 
	\begin{equation}\label{J'3}
	\begin{split}
	\beta(t)\mathcal{J}'(t)&\leq I_{\zeta}(\Sigma)-I_{\zeta}(\Sigma_t)\\
	&=\left(|\Sigma|-|\Sigma_t|\right)\inf_{M_{\zeta}}\: (\mu-|J|)+\left(|\partial\Sigma|-|\partial\Sigma_t|\right) \inf_{\partial M_{\zeta}} \left(H_{\partial M_{\zeta}} -\langle W,\nu\rangle\right)\\
	&\leq \left(|\Sigma|-|\Sigma_t|\right)\inf_{M_{\zeta}}\: (\mu-|J|).
	\end{split}
	\end{equation}where the last inequality follows from either $\partial \Sigma$ is locally length minimizing in $\partial M_{\zeta}$ and $H_{\partial M_{\zeta}} -\langle W,\nu\rangle\geq 0$ or $\inf_{\partial M_{\zeta}} \left(H_{\partial M_{\zeta}} -\langle W,\nu\rangle\right)=0$. Therefore, since $\Sigma$  is area minimizing in $M_{\zeta}$ and $\mu-|J|\geq c_1>0$, we have 
	\begin{equation}\label{J'4}
	\mathcal{J}'(t)\leq 0.
	\end{equation}Combining this with initial condition $\mathcal{J}(0)=0$, we have $\mathcal{J}(t)\leq 0$ for $t\in[0,\epsilon)$ that leads to $\theta(t)\leq 0$ for all $t\in[0,\epsilon)$. Therefore, $\theta_+(t)=0$ for all $t\in[0,\epsilon)$. In addition $\Sigma_t$, for $t\in[0,\epsilon)$, is weakly outermost free boundary stable MOTS and 
	\begin{equation}
	I_{\zeta}(\Sigma_t)=2\pi\chi(\Sigma_t).
	\end{equation}
	\noindent\emph{Step 3. Rigidity statements.} Assume either (i) or (ii) holds. By substituting $\theta_+(t)=\theta_+'(t)=0$ in \eqref{J'1}, \eqref{J'4}, and following argument of Proposition \ref{prop1rigidity}, we have for all $t\in[0,\epsilon)$ the following properties:
	\begin{enumerate}[(I)]
		\item\label{I} $|\Sigma_t|=|\Sigma|$, $J(N)=-|J|$, $\chi_{t+}=0$, $Q_t=0$, $W_t=\nabla\log\varphi_t$, and $\Pi(N_t,N_t)=\langle W_t,\nu_t \rangle$,
		\item $\mu+J(N)$ is constant on $\Sigma_t$ and equal to $\inf_{M_{\zeta}}(\mu-|J|)$,
		\item\label{III} $H_{\partial M_{\zeta}}$ is constant on $\partial M_{\zeta}$ and equal to $\inf_{\partial M_{\zeta}}H_{\partial M_{\zeta}}$,
		\item\label{IV} $\langle W_t,\nu_t \rangle$ is constant on $\partial M_{\zeta}$ and equal to $\sup_{\partial M_{\zeta}}\langle W,\nu \rangle$,
		\item $\lambda_1(\mathcal{L}_s)=0$,
		\item \label{VI} If we assume (\ref{i}), we have $|\partial\Sigma_t|=|\partial\Sigma|$,
		\item\label{VII} If we assume (\ref{ii}) we have $H_{\partial M_{\zeta}}=\langle W_t,\nu_t \rangle=\inf_{\partial M_{\zeta}}H_{\partial M_{\zeta}}=\sup_{\partial M_{\zeta}}\langle W,\nu \rangle$.
	\end{enumerate}It follows from (\ref{I}) and (\ref{VI}) that for all $t\in[0,\epsilon)$, $|\Sigma_t|$ is area minimizing in $M_{\zeta}$ and $|\partial\Sigma_t|$ is locally length minimizing in $\partial M_{\zeta}$. Therefore, mean curvatures $H(t)$ and $\langle \nabla_{T_t}N_t,T_t\rangle$ of $\Sigma_t$ in $M_{\zeta}$ and $\partial\Sigma_t$ in $\partial M_{\zeta}$, respectively, must be non-negative. By the first variation formula we have
	\begin{equation}
	0=\frac{d}{dt}|\Sigma_t|=\int_{\Sigma_t}H(t)\varphi_td\mu_{\Sigma_t},\qquad \quad 0=\frac{d}{dt}|\partial\Sigma_t|=\int_{\partial\Sigma_t}\langle \nabla_{T_t}N_t,T_t\rangle\varphi_td\mu_{\partial\Sigma_t}.
	\end{equation} Together with $\varphi_t>0$ yields to $H(t)=\langle \nabla_{T_t}N_t,T_t\rangle=0$ for all $t\in[0,\epsilon)$. This means for all $t\in[0,\epsilon)$, $\Sigma_t$ is a minimal surface in $M_{\zeta}$ and $\partial\Sigma_t$ has vanishing geodesic curvature in $\partial M_{\zeta}$ for only case (\ref{i}). Together with $\theta_+(t)=0$ and $\theta_{\pm}=\pm H(t)+\text{tr}_{\Sigma}k(t)$, we get $\theta_-(t)=\text{tr}_{\Sigma}k(t)=0$ for $t\in[0,\epsilon)$. 
	
	By assumption $\max_{\partial\Sigma}\Pi(N,N)=0$, property (\ref{I}), and (\ref{IV}), we have $\Pi(N_t,N_t)=\langle W_t,\nu_t \rangle\geq 0$ on $\partial\Sigma_t$. Moreover, using $\theta_-(t)=0$ for all $t\in[0,\epsilon)$, the first variation of $\theta_-(t)$ with speed $\varphi_-=-\varphi_t$ yields to 
	\begin{equation}\label{vartheta-}
	0=\frac{d\theta_-(t)}{dt}=\Delta_t \varphi_--2\langle W_{t-},\nabla_t \varphi_-\rangle -\left(Q_{t-}+ \text{div}_{t}W_{t-}-|W_{t-}|^2\right)\varphi_-,
	\end{equation}where 
	\begin{equation}\label{Q-}
	Q_{t-}=\frac{1}{2}R_{\Sigma_t}-\mu-J(-N)-\frac{1}{2}|\chi_{t-}|^2=-2|J|-\frac{1}{2}|\chi_{t-}|^2\leq 0,
	\end{equation} and 
	\begin{equation}\label{W_-}
	W_{t-}=\left(k(\cdot,-N_t)|_{T\Sigma_t}\right)^{\#}=-W_t=-\nabla\log\varphi_t.
	\end{equation}We used $Q_t=\chi_{t+}=0$ in \eqref{Q-} and $W_t=\nabla\log\varphi_t$ in \eqref{W_-}. Substituting \eqref{W_-} in \eqref{vartheta-} yields to 
	\begin{equation}
	-\frac{1}{2}Q_{t-}\varphi_t+\Delta_t\varphi_t+\frac{|\nabla\varphi_t|^2}{\varphi_t}=0.
	\end{equation}Integrating over $\Sigma_t$ and using boundary condition $\partial_{\nu_t}\log\varphi_t=\Pi(N_t,N_t)$, we obtain
	\begin{equation}
	\int_{\Sigma_t}\left(-|J|-\frac{1}{4}|\chi_{t-}|^2-\frac{|\nabla\varphi_t|^2}{\varphi_t}\right)\varphi_td\mu_{\Sigma_t}=\int_{\partial\Sigma_t}\Pi(N_t,N_t)\varphi_td\mu_{\partial\Sigma_t}.
	\end{equation}Since $\varphi_t>0$ on $\Sigma_t$ and $\Pi(N_t,N_t)=\langle W_t,\nu_t \rangle\geq 0$ on $\partial\Sigma_t$, we get 
	\begin{equation}\label{rigiditycon}
	\nabla_t\varphi_t=W_t=J=\chi_{t-}=0,\qquad \text{on $\Sigma_t$}\qquad \Pi(N_t,N_t)=\langle W_t,\nu_t \rangle=0\qquad \text{on $\partial\Sigma_t$}.
	\end{equation}This shows that second fundamental forms $A_t=k_{\Sigma_{t}}=0$ and $\varphi_t$ is only a function of $t$. Therefore, $(\Sigma_t,h_t)$ is totally geodesic submanifold of $M$ and $h_t$ is also independent of $t$. By setting $ds=\varphi_tdt$, we have $h_t=ds^2+g_{\Sigma}$. If we assume (\ref{i}), combining \eqref{rigiditycon}, \eqref{HdM}, and property (\ref{III}), $\partial\Sigma_t$ in $\Sigma_t$ has geodesic curvature $\kappa=\inf_{\partial M_{\zeta}}H_{\partial M_{\zeta}}$. If we assume (\ref{ii}), it follows from \eqref{rigiditycon} and property (\ref{VII}) that $\kappa=H_{\partial M_{\zeta}}=0$. Finally, by boundary condition \eqref{rigiditycon} and Theorem \ref{ThmA1}, we have $\lambda_1(\mathcal{L})=0$ on $\Sigma_t.$ This complete the proof.
	
\end{proof}

\subsection{Free Boundary Marginally Outer Trapped Surfaces with Low Index}
One can define a notion of Morse index $i_s(\Sigma,\partial \Sigma)$ for free boundary MOTS as follows. Recall that the self-adjoint operator $\mathcal{L}_s=(L_s,B_s)$ as follows.
\begin{equation}
L_s(\psi)= -\Delta_{\Sigma}\psi +Q\psi, \: \: \text{on}\: \: \Sigma \hspace{0.5in} B_s(\psi)= \frac{\partial \psi}{\partial \nu}-(q-\langle W,\nu \rangle)\psi \: \: \text{on} \: \: \partial \Sigma.
\end{equation}
This yields a positive semi-definite index form
	\begin{equation}
	P(u,u):=\int_{\Sigma} \left(|\nabla u|^2 +Qu^2\right)d\mu_{\Sigma}-\int_{\partial \Sigma} \left(q-\langle W,\nu\rangle\right) u^2 d\mu_{\partial\Sigma} \geq 0.
	\end{equation}
One can define $i_s(\Sigma,\partial \Sigma)$ to be the Morse index associated with the operator $\mathcal{L}_s$; that is, the number of negative eigenvalues of the associated bilinear form. Note that a function $f\in W^{1,2}(\Sigma,\mathbb{R})$ is an eigenfunction of the index form $P(\cdot,\cdot)$ with eigenvalue $\lambda_s$ if $P(f,g)=\lambda_s (f,g)_{L^2}$ for all $g\in W^{1,2}(\Sigma,\mathbb{R})$.
\begin{remark}
Note that by the arguments of Section 5, if $q\leq 0$, then $\lambda(\mathcal{L})\leq \lambda(\mathcal{L}_s)$ and thus $i_s(\Sigma,\partial \Sigma)$ gives a kind of upper bound definition. It is possible to define the analogous lower bound definition based on a different symmetrized operator, $\mathcal{L}_z$, which has the property that, if $q\leq 0$, then $\lambda(\mathcal{L}_z)\leq \lambda (\mathcal{L})$, cf. \cite{AnderssonMarsSimon}. However, the operator $\mathcal{L}_z$ involves extra $W$ terms which make for more heavy handed statements.
\end{remark} 
Suppose now that $h$ is an eigenfunction of $\mathcal{L}_s$. In that case we have the following lemma of \cite{CFP12}, itself based on \cite{LY82}.
\begin{lemma}\label{lemcfp}\cite{CFP12}
There exists a conformal map $f:\Sigma \to S^2$ such that $\int_{\Sigma}fhdA_{\Sigma}=0$ and $f$ has degree $\leq \left[ \frac{g+3}{2} \right] $.
\end{lemma}
As a consequence, one obtains the following characterisation of free boundary MOTS with $i^s(\Sigma,\partial \Sigma)=1$.
\begin{proposition}\label{indexestimate}
Let $(M,g,k)$ be an initial data set with smooth boundary $\partial M$. Let $(\Sigma,\partial \Sigma)$ be a compact orientable two-sided surface of genus $g$ with $l\geq 1$ boundary components. Suppose that $(\Sigma,\partial \Sigma)$ is a free boundary MOTS with $H_{\partial M}-\langle W,\nu \rangle \geq 0$. If $\mu-|J|\geq 0$, then $i_s(\Sigma, \partial \Sigma)=1$ implies that $l< 10$ if $g$ is even and $l< 14$ if $g$ is odd; and if $\mu-|J|\geq c>0$, then $i_s(\Sigma, \partial \Sigma)=1$ implies $|\Sigma|\leq \frac{2\pi\left(7-(-1)^g-l\right)}{c}$.
\end{proposition}

\begin{proof} The proof follows from Theorem 1.2. of \cite{CFP12} and since it is short we include it here. Let $h\geq 0$ be a first eigenfunction of the index form $P(\cdot,\cdot)$. By Lemma \ref{lemcfp}, there exists a conformal map $f:\Sigma\to S^2$ of degree $\leq \left[ \frac{g+3}{2}\right]$ such that $\int_{\Sigma}fh dA_{\Sigma}=0$. This shows that the components $f_i$ of $f$ are orthogonal to $h$. Therefore, $f$ is an basis element of orthogonal of index. Consequently 
\begin{equation}
P(f_i,f_i)=\int_{\Sigma} \left(|\nabla f_i|^2 +Qf_i^2\right)d\mu_{\Sigma}-\int_{\partial \Sigma} \left(q-\langle W,\nu\rangle\right) f_i^2 d\mu_{\partial\Sigma} \geq 0
\end{equation}
Summing over $i$ and using $\sum_{i=1}^{3}|f_i|^2=1$ gives 
\begin{equation}
\int_{\Sigma} \left(|\nabla f_i|^2 +Q\right)d\mu_{\Sigma}-\int_{\partial \Sigma} \left(q-\langle W,\nu\rangle\right) d\mu_{\partial\Sigma} \geq 0
\end{equation} Let $\overline{\Sigma}$ be a compact domain with genus $g$ by gluing a disk on each boundary component of $\Sigma$. Since $f:\overline{\Sigma}\to S^2$ is conformal we have 
\begin{equation}
\int_{\Sigma}|\nabla f|^2d\mu_{\Sigma}< \int_{\overline{\Sigma}}|\nabla f|^2d\mu_{\overline{\Sigma}}=2|f(\overline{\Sigma})|=2|S^2|\text{deg}(f)\leq 8\pi \left[ \frac{g+3}{2}\right]=8\pi\left(\frac{g+3-\frac{1+(-1)^g}{2}}{2}\right)
\end{equation}
where $|f(\overline{\Sigma})|$ is the area of the image of $\overline{\Sigma}$, and thus
\begin{equation}
8\pi \left[ \frac{g+3}{2}\right] - \int_{\partial \Sigma} \left(q-\langle W,\nu\rangle\right) d\mu_{\partial\Sigma} > - \int_{\Sigma}Q d\mu_{\Sigma}
\end{equation}
Using $Q=K_{\Sigma}-\mu -J(N)-\frac{1}{2}|\chi_+|^2$ and the Gauss-Bonnet theorem
\begin{equation}
\int_{\Sigma}K_{\Sigma}dA_{\Sigma} +\int_{\partial \Sigma} k d\mu_{\partial \Sigma}=2\pi \chi(\Sigma)=2\pi(2-2g-l)
\end{equation}
we obtain 
\begin{equation}
8\pi \left[ \frac{g+3}{2}\right] -\int_{\partial \Sigma} \left(q-\langle W,\nu\rangle\right) d\mu_{\partial \Sigma} -\int_{\partial \Sigma} k d\mu_{\partial \Sigma}   
> -2\pi(2-2g-l) + \int_{\Sigma} \left(\mu+J(N)+\frac{1}{2}|\chi_+|^2\right)d\mu_{\Sigma} 
\end{equation}
By the free boundary property we get
\begin{equation}\label{in1}
8\pi \left[ \frac{g+3}{2}\right] -\int_{\partial \Sigma} \left(H_{\partial M}-\langle W,\nu\rangle\right) d\mu_{\partial \Sigma} +2\pi(2-2g-l)> \int_{\Sigma} \left(\mu+J(N)+\frac{1}{2}|\chi_+|^2\right)d\mu_{\Sigma} 
\end{equation}
Using $\mu -|J|\geq 0$ and $H_{\partial M}-\langle W,\nu \rangle \geq 0$ we have 
\begin{equation}
2\left[ \frac{g+3}{2}\right]>  \frac{1}{2}(g +\frac{l}{2} -1) 
\end{equation}
from which the conclusion follows. Finally, if $\mu-|J|\geq c$, then \eqref{in1} leads to 
\begin{equation}
8\pi \left[ \frac{g+3}{2}\right]  +2\pi(2-2g-l)> c |\Sigma| 
\end{equation}
which finishes the proof.

\end{proof}

\subsection{A Diameter Estimate for Stable Free Boundary Marginally Outer Trapped Surfaces}
The proof of the diameter estimate in \cite{CF19} combines an argument of Fischer-Colbrie \cite{F85}, a classical theorem of Hartman \cite{H64}, and a computation in \cite{SY83} employed in \cite{C17}. Those arguments apply in this setting to yield the following.
\begin{proposition}\label{prop6.7}
Let $(M,g,k)$ be an initial data set with boundary $\partial M$ having second fundamental form $\Pi(\cdot, \cdot)$ and outwards pointing normal $\nu$. Let $\Sigma$ be a complete, two-sided, connected, embedded stable free boundary MOTS with unit normal $N$. Suppose on $\partial \Sigma$ that $\Pi(N,N)\leq 0$. Finally, assume that either one of the following holds
\begin{itemize}
\item[(i)] $\inf_{M}(\mu-|J|)>0$ and $\inf_{\partial M}\left(H_{\partial M}-\langle W,\nu \rangle \right) \geq 0 $ with no components having $H_{\partial M}-\langle W,\nu \rangle=0$,
\item[(ii)] $\inf_{M}(\mu-|J|)$ is non-negative and $\partial M$ is such that $H_{\partial M}-\langle W,\nu \rangle >0$.
\end{itemize}
Then $\Sigma$ is compact and satisfies the intrinsic diameter estimate
\begin{equation}
\emph{diam}(\Sigma)\equiv \sup_{x,y\in \Sigma}d_{\Sigma}(x,y)\leq \emph{min}\left( \frac{2\pi}{\sqrt{3\inf_{M}(\mu-|J|)}},\frac{\pi +\frac{8}{3}}{\inf_{\partial M}\left(H_{\partial M}-\langle W,\nu \rangle\right)}\right),
\end{equation} 
and moreover one has 
\begin{equation}
0<\inf_{M}\left(\mu-|J|\right)\mathscr{H}^2(\Sigma)+\inf_{\partial M}\left(H_{\partial M}-\langle W,\nu \rangle\right)\mathscr{H}^1(\partial \Sigma)\leq 2\pi \chi(\Sigma),
\end{equation} 
where $\mathscr{H}^{1}(\partial \Sigma)$ and $\mathscr{H}^{2}(\Sigma)$ denote the $1$ and $2$ dimensional Hausdorff measures of $\partial \Sigma$ and $\Sigma$, respectively. Thus $\Sigma$ is diffeomorphic to a disc.
\end{proposition}
\begin{proof}The proof is long but it follows directly from the proof of Proposition 1.8 in \cite{CF19}.
\end{proof}

\appendix

\section{Linear Elliptic PDE with Neumann Boundary Condition}
We include the proof of the following theorem and it is a direct consequence of Dirichlet boundary condition in \cite{Evans} .

\begin{theorem}\label{ThmA1}
Let $(\Omega,g)$ be a $C^{2,\alpha}$ Riemannian manifold with boundary $\partial\Omega$. Consider the differential equation 
\begin{equation}
Lu:=-\Delta_gu+b(x)^i\nabla_i u+c(x)u\quad \Omega,\qquad Bu:=\frac{\partial u}{\partial \nu}+\beta(x)u\qquad \partial\Omega,
\end{equation} where $b,c\in C^{\alpha}(\bar{\Omega})$, $\beta\in C^{1,\alpha}(\partial\Omega)$. If $\beta\geq 0$ and $Bu=0$, then there exist a simple principle real eigenvalue $\lambda_1$. Moreover, for any eigenvalue $\mu$, we have $\text{Re}\mu\geq \lambda_1$. The corresponding eigenfunction $\varphi$, $L\varphi=\lambda_1\varphi$ is unique up to a multiplicative constant and can be chosen to be real and everywhere positive. The adjoint $L^*$ (with respect to the $L^2$ inner product) has the same principle eigenvalue.
\end{theorem}
\begin{proof}
Let $\delta>\sup_{\Omega}-c(x)$ and define an operator $L_{\delta}=L+\delta$ such that
\begin{equation}
L_{\delta}u:=-\Delta_gu+b(x)^i\nabla_i u+(c(x)+\delta)u\quad \Omega,\qquad\qquad Bu:=\frac{\partial u}{\partial \nu}+\beta(x)u\qquad \partial\Omega.
\end{equation}Since $c(x)+\delta\geq 0$, by Theorem 6.31 of \cite{GT} the PDE $L_{\delta}u=f$ with $Bu=\gamma(x)$ has a unique $u\in C^{2,\alpha}$ solution for all $f\in C^{\alpha}(\Omega)$ and $\gamma\in C^{1,\alpha}(\partial\Omega)$. In particular, we have 
\begin{equation}
|u|_{C^{2,\alpha}(\Omega)}\leq C\left(|u|_{C^{\alpha}({\Omega})}+|\gamma|_{C^{1,\alpha}(\partial\Omega)}+|f|_{C^{\alpha}(\Omega)}\right)
\end{equation}Let $K$ be  the set of non-negative functions in $X=C^{2,\alpha}(\Omega)$. Then we define a compact linear operator $A:X\rightarrow X$ such that $Af=u$, where $u$ is a solution of $L_{\delta}u=f$ with boundary condition $Bu=0$. If $f\in K$, then we claim $u=Af\geq 0$. Assume $u$ can be negative. By the weak maximum principle we know the minimum of $u$ is at boundary $x_0\in\partial\Omega$ and $u(x_0)<0$. Since $Bu\geq 0$, we have $\frac{\partial u}{\partial\nu}(x_0)\geq 0$ which contradicts the Hopf maximum principle. If $f\in K\,\backslash\, \{0\}$, then we show $u=Af>0$ on $\Omega$. Assume there is a point $x_0\in\partial\Omega$ such that $u(x_0)=0$. This means $x_0$ is the minimum point and therefore by the Hopf maximum principle we have $\frac{\partial u}{\partial\nu}(x_0)< 0$ which contradicts the boundary condition. Therefore, $A$ is a positive operator and by Krein--Rutman theorem, there exists a unique non-negative
function $\varphi\in C^{\alpha}(\Omega)$ such that $A\varphi=\xi\varphi$ for positive real eigenvalue $\xi$. Moreover, positivity of $\varphi\in C^{2,\alpha}(\bar{\Omega})$ follows from positivity of $A$. The elliptic regularity implies that $\varphi$ is in fact smooth if the coefficients of $L$, $f$, and $\gamma$ are smooth. It follows that $L\varphi=(\xi^{-1}-\delta)\varphi$ on $\Omega$ and $B\varphi=0$ on $\partial\Omega$, where $\lambda_1=(\xi^{-1}-\delta)$ is a real eigenvalue and $\varphi$ is a positive eigenfunction. 

To show that $\text{Re}\mu\geq \lambda$, we follow Section 6.5.2 of \cite{Evans}. Let $\psi$ be a (possibly complex) eigenfunction of $L$ with eigenvalue $\mu$. Define $u=\varphi^{-1}\psi$ and a direct computation gives
\begin{equation}
-\Delta_g u+b_1^{i}(x)\nabla_i u+\left(\lambda_1-\mu\right) u=0
\end{equation}where $b_1^{i}(x)=b^{i}(x)-2\varphi^{-1}\nabla^i\varphi$. Using also the complex conjugate of above equation a short calculation, see Section 6.5.2 of Evans, gives
\begin{equation}
K(|u|^2)\leq 2\left(\text{Re}\mu-\lambda\right)|u|^2
\end{equation} where $K:=-\Delta_g+b_1^{i}(x)\nabla_i$. Moreover, using $B\varphi=B\psi=0$, we have the boundary condition 
\begin{equation}
\frac{\partial}{\partial\nu} u=-\varphi^{-2}\psi\frac{\partial}{\partial\nu}\varphi+\varphi^{-1}\frac{\partial}{\partial\nu}\psi=u\beta-u\beta=0
\end{equation}which implies
\begin{equation}\label{BCu2}
\begin{split}
\frac{\partial}{\partial\nu}|u|^2=\langle\frac{\partial}{\partial\nu} u,\bar{u}\rangle+\langle u,\frac{\partial}{\partial\nu}\bar{u}\rangle=0
\end{split}
\end{equation}Therefore, if $\text{Re}\mu-\lambda_1< 0$, then by the weak maximum principle and the Hopf maximum principle, the maximum of $|u|^2$ is at the boundary $x_0\in\partial\Omega$ and $\frac{\partial |u|^2}{\partial\nu}(x_0)>0$. This contradicts equation \eqref{BCu2} and $\text{Re}\mu\geq \lambda_1 $ as claimed.
Moreover, clearly if $Bu=0$, we have adjoint operator $L^*v=\lambda^*v$ with boundary $Bv=0$ such that
\begin{equation}
0=\langle Lu,v\rangle-\langle u,L^*v\rangle=(\lambda-\lambda^*)\langle u,v\rangle.
\end{equation}Since real positive eigenfunctions $u$ and $v$ cannot be orthogonal, we have $\lambda=\lambda^*$.
	\end{proof}

\end{document}